\documentclass{amsart}

\usepackage{hyperref}
\usepackage{float}
\hypersetup{
    colorlinks=true,
    citecolor=blue,
    linkcolor=blue,
    filecolor=magenta,      
    urlcolor=cyan,
    pdftitle={Categorified Wrapping Number Conjecture},
    pdfpagemode=FullScreen,
    }

\usepackage[margin=1in]{geometry}

\usepackage{amsmath,amssymb,amsthm, mathrsfs}
\usepackage{dsfont}
\usepackage[dvipsnames]{xcolor}
\usepackage{tikz}
    \usetikzlibrary{cd}
    \usetikzlibrary {arrows.meta}
\usepackage{tikz-cd}
\allowdisplaybreaks
\usepackage{standalone}
\usepackage{comment}
\usepackage{wasysym} 

\usepackage{mathtools} 
\usepackage{bbm} 

\theoremstyle{plain}
\newtheorem{theorem}{{Theorem}} 
\newtheorem{lemma}[theorem]{{Lemma}}
\newtheorem{proposition}[theorem]{{Proposition}}
\newtheorem{corollary}[theorem]{{Corollary}}

\newtheorem{conjecture}[theorem]{{Conjecture}}

\newtheorem{question}[theorem]{{Question}}

\newtheorem{remark}[theorem]{{Remark}}

\theoremstyle{definition}
\newtheorem{definition}[theorem]{{Definition}}
\newtheorem{notation}[theorem]{{Notation}}

\newcommand{\C}{\mathbb{C}}

\newcommand{\R}{\mathbb{R}}
\newcommand{\Z}{\mathbb{Z}}
\newcommand{\essl}{\mathfrak{sl}}

\newcommand{\CAKh}{\mathrm{CAKh}} 
\newcommand{\AKh}{\mathrm{AKh}}

\newcommand{\rank}{\mathrm{rank}~}

\newcommand{\wrap}{\mathrm{wrap}} 
\newcommand{\ww}{\mathbb{W}}
\newcommand{\vv}{\mathbb{V}}

\newcommand{\grk}{\mathrm{gr}_k}
\newcommand{\dep}{\mathrm{Dep}}

\newcommand{\kerat}{\mathcal{K}}
\newcommand{\imgat}{\mathcal{I}}
\newcommand{\wmerge}{\mathcal{W} \sqcup \mathcal{W} \rightarrow \mathcal{W}}
\newcommand{\wsplit}{\mathcal{W} \rightarrow \mathcal{W} \sqcup \mathcal{W}}
\newcommand{\veemerge}{\mathcal{V} \sqcup \mathcal{V} \rightarrow \mathcal{W}}
\newcommand{\veesplit}{\mathcal{W} \rightarrow \mathcal{V} \sqcup \mathcal{V}}
\newcommand{\mixmerge}{\mathcal{V} \sqcup \mathcal{W} \rightarrow \mathcal{V}}
\newcommand{\mixsplit}{\mathcal{V} \rightarrow \mathcal{V} \sqcup \mathcal{W}}
\newcommand{\KBSM}{\mathrm{KBSM}} 
\newcommand{\gr}{\mathrm{gr}} 

\newcommand{\crossing}{
    \begin{tikzpicture}[scale=.25, baseline=-.67ex]
        \draw (-1,1) -- (1,-1);
        \filldraw[white] (0,0) circle (.4cm);
        \draw (-1,-1) -- (1,1);
    \end{tikzpicture}
}
\newcommand{\vertres}{
    \begin{tikzpicture}[scale=.25, baseline=-.67ex]
        \draw (-1,-1) .. controls (0,0) .. (-1,1);
        \draw (1,-1) .. controls (0,0) .. (1,1);
    \end{tikzpicture}
}
\newcommand{\horizres}{
    \begin{tikzpicture}[scale=.25, baseline=-.67ex]
        \begin{scope}[rotate=90]
            \draw (-1,-1) .. controls (0,0) .. (-1,1);
            \draw (1,-1) .. controls (0,0) .. (1,1);
        \end{scope}
    \end{tikzpicture}
}

\usepackage{pinlabel}
\usepackage{standalone}

\usepackage{subcaption} 
\usepackage{todonotes}

\author[B.~Daniels]{Benjamin Daniels}
\address{Department of Mathematics, UC Davis, One Shields Ave., Davis, CA 95616-8633, U.S.A.
}
\email{\href{mailto:bkdaniels@ucdavis.edu}{bkdaniels@ucdavis.edu}}

\author[M.~Zhang]{Melissa Zhang}
\address{Department of Mathematics, UC Davis, One Shields Ave., Davis, CA 95616-8633, U.S.A.
}
\email{\href{mailto:mlzhang@ucdavis.edu}{mlzhang@ucdavis.edu}}

\title{On the Categorified Wrapping Number Conjecture}
\date{\today}

\begin{document}

\begin{abstract}
    We prove the Categorified Wrapping Number Conjecture for large classes of annular links, including alternating annular links and tangle closures exhibiting plumbed link phenomena. 
    We do so by characterizing when a resolution is sufficient to produce a nonzero homology class in $k$-grading $\wrap(L)$ on its own. 
    This characterization primarily concerns the type of crossing resolutions abutting trivial circles.   
\end{abstract}

\maketitle

\tableofcontents

\section{Introduction} 
\label{sec:intro}

    In 2000, Khovanov categorified the Jones polynomial link invariant \cite{Khovanov-original}, sparking a revolution in the development of link homologies. 
    In the following years, his theory has been adapted and generalized in numerous ways (see \cite{BN-tanglesandcobs}, \cite{Lee-leehom}, and \cite{ras-s} for some highly impactful early constructions). 
    
    One such generalization is Asaeda--Przytycki--Sikora's extension of Khovanov homology to links in $I$-bundles over surfaces \cite{APS-khovanovhomologyoversurfaces}, where $I$ is the unit interval. This homology theory is sometimes referred to as \emph{APS homology} in reference to the authors' names.
    The most popular version of APS homology is \emph{annular Khovanov homology}, which provides an invariant for \emph{annular links}, i.e.\ links in the thickened annulus.
    
    Let $A = \R^2 - \{0\}$, and let $L \subset A \times I$ be an annular link. 
    The \emph{annular Khovanov homology} of $L$, denoted by $\AKh(L)$, is a triply-graded homology-type invariant of $L$ computed from an \emph{annular diagram} for $L$, i.e.\ a link diagram that misses the origin, which we will denote by an asterisk (*) in the plane. 
    In addition to the homological and quantum gradings inherited from Khovanov homology, $\AKh$ has an additional `$k$-grading' determined by the location of the $*$: this endows the usual Khovanov complex with a $k$-filtration grading, and $\AKh$ is the homology of the associated graded object.

    Annular Khovanov homology has attracted considerable interest over the past decades.
    For example, $\AKh$ admits spectral sequences to gauge and Floer theories (see \cite{Roberts-akhfloer} and \cite{Yi-instantakh}  ), enjoys an $\mathfrak{sl}_2(\C)$ action (see \cite{GLW-sl2action} and \cite{Kim-caltechsl2thesis}), and has been extended to equivariant settings (see \cite{BPW-AKhq} and \cite{Akhmechet-equivariant}). 
   
    As a categorification of the Kauffman bracket skein module of the annulus $\KBSM(A)$ \cite{Prz-skein}, annular Khovanov homology inherits long-open questions from skein theory, including the conjecture at the center of the present article. 
    The \emph{wrapping number} of an annular link $L$, denoted by $\wrap(L)$, is the minimal \emph{geometric} intersection between a link $L$ and a meridional disk in $A \times I$. 
    This quantity is related to the annular \emph{winding number grading}, or \emph{$k$-grading}, in both $\KBSM(A)$ and $\AKh$, which keeps track of the \emph{algebraic} wrapping of each Kauffman state around the $*$, in a sense.
    In particular, $\AKh(L)$ is supported in $k$-gradings with absolute value $\leq \wrap(L)$ \cite{GLW-sl2action}.

    This naturally leads to the following conjecture of Grigsby from 2010 \footnote{See the anecdote under Conjecture 1.2 on page 2 of \cite{Martin-wrapping}.}:

    \begin{conjecture}{[Grigsby, MSRI 2010]}
        \label{con:catwrapcon}
        (Categorified Wrapping Number Conjecture)
        The annular Khovanov homology of an annular link $L$ is nontrivial in $k$-grading $\wrap(L)$.  
    \end{conjecture}

    We direct the reader to \cite{Kim-caltechsl2thesis} for a summary of present bounds and attempts to prove Conjecture \ref{con:catwrapcon}.
    Grigsby's conjecture is motivated by the following older conjecture, first posed in a seminal paper by Hoste and Przytycki in 1995: 

    \begin{conjecture}{\cite{HostePrz-skeinmoduleofwhiteheadmanifolds}}
        \label{con:wrappingcon}
        (Wrapping Number Conjecture)
        The Kauffman bracket skein module of an annular link $L$ is nonzero in annular degree $\wrap(L)$. 
    \end{conjecture}   

    In support of their conjecture, Hoste--Przytycki proved that Conjecture \ref{con:wrappingcon} holds for \emph{adequately wrapped} annular links. 
    An adequately wrapped annular link enjoys diagrams with particular resolutions, ones which guarantee that the link satisfies Conjecture \ref{con:wrappingcon}. 

    In this paper, we generalize Hoste--Przytycki's results in the categorified setting by identifying a broader class of resolutions which may be associated to a nonvanishing annular Khovanov homology class.
    We refer to these resolutions as \textit{perfectly wrapped} and \textit{uniform}. 
    The property of being perfectly wrapped limits the type of cobordisms entering and exiting the resolution.
    Uniformity is a condition on trivial circles; it requires each trivial circle to abut only $0$-resolutions of crossings, or only $1$-resolutions of crossings.

    An abridged version of our main theorem, Theorem \ref{thm:perfectlywrappeduniform}, is the following:

    \begin{theorem}
        \label{thm:mainthm}
        Let $L$ be an annular link with a diagram $D$ admitting a perfectly wrapped uniform resolution $D_u$. 
        Then $\AKh(L)$ is nontrivial in $k$-grading $\wrap(L)$. 
    \end{theorem}
    
    In particular, every perfectly wrapped uniform resolution $D_u$ produces at least one distinguished generator $x \in \CAKh(D_u)$ such that 
        \[(0, ..., 0, x, 0, ..., 0)\]
    generates a nonzero annular Khovanov homology class in $\wrap(L)$ $k$-grading.

    Theorem \ref{thm:mainthm} follows from the data of a single resolution and the cobordisms entering and exiting it. 
    Namely, we look at the distinguished generators which are always sent to zero by $\wrap(L)$ $k$-graded differential components, and the distinguished generators which are never mapped to by $\wrap(L)$ $k$-graded differential components.

    Theorem \ref{thm:mainthm} encompasses Hoste--Przytycki's adequately wrapped links, along with a number of other annular link classes, chief among them alternating annular links. 

    The proof for alternating annular links requires a transformation between another type of resolution and perfectly wrapped uniform resolutions.
    This former type is called \textit{almost uniform} as opposed to uniform. 
    It has an alternate condition on its trivial circles: Each trivial circle in an almost uniform resolution abuts a singular kind of crossing resolution on its exterior, and the opposite kind on its interior.
    
    The definition is motivated entirely by Lemma \ref{lma:alternatinglinksalmostuniform}, which gives a perfectly wrapped almost uniform resolution for every alternating diagram for an annular link. 
    The following theorem and corollary, adapted from Theorem \ref{thm:expandingclass} and Corollary \ref{cor:alternatinglinkssatisfycwnc}, completes the proof. 
    They are proven entirely in the setting of resolutions.

    \begin{theorem}
        \label{thm:thmforalternating}
        Let $D$ be a diagram for an annular link $L$, with no nugatory crossings (understood in the annular setting).
        Suppose that $D$ admits a perfectly wrapped almost uniform resolution $D_u$.
        Then $D$ admits a perfectly wrapped uniform resolution $D_v$. 
    \end{theorem}

    \begin{corollary}
        \label{cor:alternatinglinkssatisfycatwrapcon}
        The Categorified Wrapping Number Conjecture holds for every alternating annular link.
    \end{corollary}
    
    Beyond alternating annular links, we examine phenomena associated with plumbed links. 
    The following takes from Corollaries \ref{cor:braidsandchains} and \ref{cor:circlecable}:

    \begin{corollary}
        \label{cor:buildingclasseswithpwur}
        The annular closures of the following classes of tangles satisfy the conditions of Theorem \ref{thm:mainthm}, and therefore Conjecture \ref{con:catwrapcon}: 

        \begin{enumerate}
            \item A braid (drawn vertically) with some crossings replaced by horizontal 2-braids;
            in other words, the result of surgering a braid along some twisted horizontal bands. See Figure \ref{fig:braid-chains} for an example.
            \item A composition of braids and identity braids wearing belts. See Figure \ref{fig:plumby} for an example.
        \end{enumerate}
        
    \end{corollary}

    \begin{figure}
        \centering
        \begin{subfigure}[t]{0.5\textwidth}
            \centering
            \includegraphics[width=0.3\linewidth]{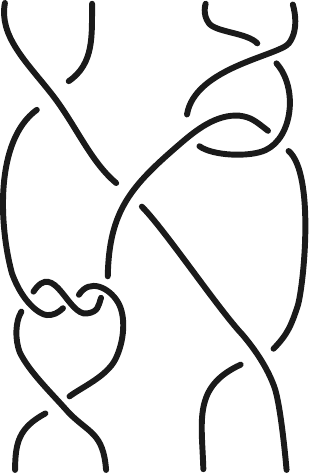}
            \caption{A 4-braid with inserted horizontal 2-braids (`chains') inserted in two locations.}
            \label{fig:braid-chains}
        \end{subfigure}%
        ~
        \begin{subfigure}[t]{0.5\textwidth}
            \centering
            \includegraphics[width=0.3\linewidth]{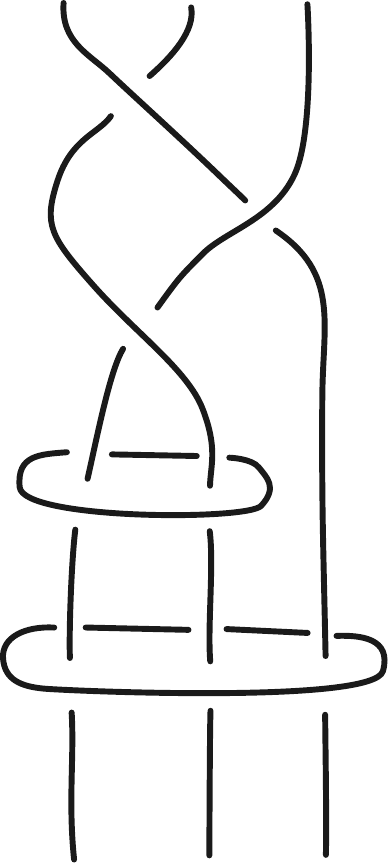}
            \caption{A braid wearing unknotted belts.}
            \label{fig:plumby}
        \end{subfigure}
        \caption{Annular closures of these tangles are examples of the two classes of annular links from Corollary \ref{cor:buildingclasseswithpwur}.}
        \label{fig:buildingclasseswithpwur}
    \end{figure}

    The proof of Corollary \ref{cor:buildingclasseswithpwur} relies on work of Grigsby--Ni, who showed that $\AKh$ over $\C$ coefficients is 1-dimensional at $k$-grading $\wrap(L)$ for braid closures \cite{GrElYi-braiddetection}.
    We note the similarities between these types and work of Martin; she found additional classes of adequately wrapped annular links, proving both Conjecture \ref{con:catwrapcon} and Conjecture \ref{con:wrappingcon} for these classes \cite{Martin-wrapping}. 
    
    Hoste--Przytycki used the fact that the blackboard-framed $n$-cabling of an adequately wrapped diagram for an annular link is also adequately wrapped \cite{HostePrz-skeinmoduleofwhiteheadmanifolds}. 
    We obtain an analogous result, with an additional observation about adding unlink components to diagrams.
    The following is adapted from Corollaries \ref{cor:cablingop} and \ref{cor:loopinsertion}.

    \begin{corollary}
        \label{cor:linkopsobtainingmorepwur}
        Suppose that $D$ is a diagram for an annular link $L$ that admits a perfectly wrapped uniform resolution. Then

        \begin{enumerate}
            \item The blackboard-framed $n$-cabling of $D$ admits a perfectly wrapped uniform resolution (see Figure \ref{fig:square-blackboard-cable} for an example), and
            \item The addition of an arbitrary number of disjoint loops (`earrings'), each wrapping around a single strand, to $D$ produces a diagram with a perfectly wrapped uniform resolution (see Figure \ref{fig:trefoil-earrings} for an example).  
        \end{enumerate}
        
    \end{corollary}

    \begin{figure}
        \centering
        \begin{subfigure}[t]{0.5\textwidth}
            \centering
            \labellist
                \pinlabel {*} at 150 90
            \endlabellist
            \includegraphics[width=0.8\linewidth]{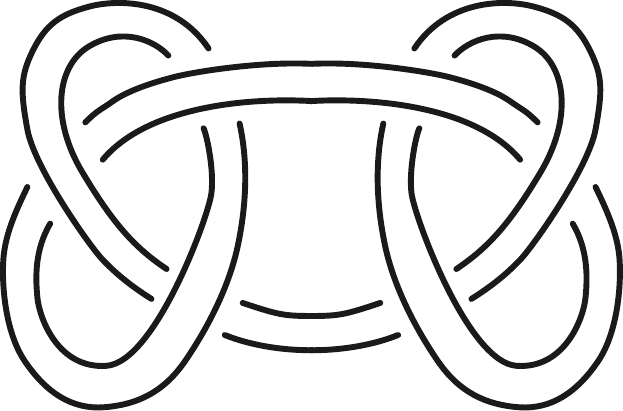}
            \caption{A blackboard-framed 2-cabling of a square knot in the thickened annulus.}
            \label{fig:square-blackboard-cable}
        \end{subfigure}
        ~
        \begin{subfigure}[t]{0.5\textwidth}
            \centering
            \labellist
                \pinlabel {*} at 80 90
            \endlabellist
            \includegraphics[width=0.6\linewidth]{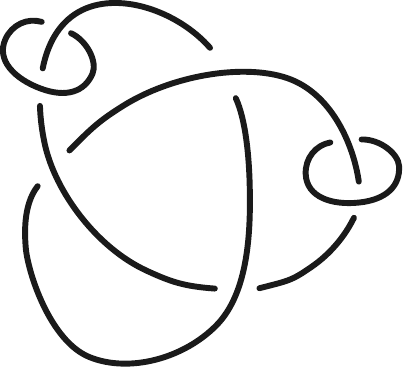}
            \caption{An annular trefoil wearing earrings.}
            \label{fig:trefoil-earrings}
        \end{subfigure}
        \caption{Examples of operations from Corollary \ref{cor:linkopsobtainingmorepwur}.}
        \label{fig:linkopsobtainingmorepwur}
    \end{figure}
    
    It remains to be seen how far this approach can be taken. 
    In one direction, Theorem \ref{thm:mainthm} could likely be upgraded to prove Conjecture \ref{con:wrappingcon} for certain annular links (perhaps without additional conditions).
    In another direction, we are motivated to ask about the extent to which links admit perfectly wrapped uniform resolutions.
    
    \begin{question}
        \label{que:everylinkhaspwur}
        Does every annular link admit a diagram with a perfectly wrapped uniform resolution? 
    \end{question}

    If true, Question \ref{que:everylinkhaspwur} would imply Conjecture \ref{con:catwrapcon} by Theorem \ref{thm:mainthm}. 
    It would also imply that the rank of the annular Khovanov homology in $k$-grading $\wrap(L)$ (over $\Z$ coefficients) is nonzero in general.

    In light of the above, it is feasible that some annular link will not admit a diagram with a perfectly wrapped uniform resolution; we may have a link whose annular Khovanov homology in $k$-grading $\wrap(L)$ is a direct sum of torsion groups.

    We can also ask about the maximal number of perfectly wrapped uniform resolutions an annular link may admit for a given diagram. 
    Denote this quantity by $U(L)$ in the manner of Corollary \ref{cor:lowerboundforwraplrank}.
    
    \begin{question}
        \label{que:computingrl}
        Can we identify, for any annular link $L$, a class of diagrams $D$ for $L$ such that the number of perfectly wrapped uniform resolutions of $D$ is equal to $U(L)$? 
    \end{question}

    The authors believe that this equality holds when $D$ is a minimal crossing diagram. 

\subsection{Acknowledgements} \label{subsec:acknowledgements}

    Special thanks to Jozef Przytycki for providing the authors access to his manuscript. We also thank Wren Burrill and Nathan Singh for useful conversations.

\section{Annular Khovanov homology} \label{sec:annularkhovanovhomology}

    Throughout the article, we work over $\Z$ coefficients.
    We will indicate when we use $\C$ coefficients as appropriate, e.g.\ in reference to the $\essl_2(\C)$ action on $\AKh$ over $\C$.

\subsection{Cube of resolutions} \label{subsec:cubeofres}

    We recount the method of computation for annular Khovanov homology. Let $L \subseteq A \times I$ be an annular link, presented via an annular diagram $D$ with $n$ crossings: $n_+$ positive and $n_-$ negative. We now form the \textit{cube of resolutions} for $D$. 
    
    Start with a directed graph $\Gamma$, where each vertex is some binary string $u \in \{0,1\}^n$. Let $|u|$ denote the sum of each digit in $u$; i.e., the number of $1$'s in $u$. Then construct an edge $u \rightarrow v$ if and only if $|u|  = |v|-1$, written $u \prec_1 v$. An example for $n = 3$ is given below:
    
    \begin{center}
        \begin{tikzcd}
                                    & 100 \arrow[r] \arrow[rd]  & 110 \arrow[rd] &     \\
    000 \arrow[ru] \arrow[r] \arrow[rd] & 010 \arrow[ru] \arrow[rd] & 101 \arrow[r]  &      111 \\
                                    & 001 \arrow[ru] \arrow[r]  & 011 \arrow[ru] &    
        \end{tikzcd}.
    \end{center}

    Number the crossings in $D$. Then we assign to any given binary string $u$ in $\Gamma$ a complete resolution of $D$, where the $i$th crossing is resolved according to the $i$th digit of $u$. The two types of resolutions are as follows: 
    
    \[
        \vertres 
        \xleftarrow{{\color{red} 0}} 
        \crossing
        \xrightarrow{{\color{blue} 1}} 
        \horizres.
    \]

    We then refer to the complete resolution of $D$ as $D_u$, or the complete resolution of $D$ corresponding to the binary string $u$. In general, $D_u$ will consist of a collection of planar circles of two types: \textit{trivial} and \textit{nontrivial}. The former class of circles are those which do not wind around the $*$, while the latter class are those which do. 

    There is now the question of what topological data the edges should encode. As each edge presently represents a bit flip, the resolution at the source of an edge will differ from the resolution at the target by a single crossing resolution swap, from type $0$ to type $1$ as above. Enacting this swap can affect the planar circles in the source resolution in six different ways, as shown in Figure \ref{fig:akh-diffs}.

    \begin{figure}
        \begin{tabular}{|c|c|}
        \hline
        merges & splits 
        \\ \hline 
        \begin{tikzpicture}[scale=.5]
\def \r{1};
\def \R{2};
\def \m{1.5}; 
\def \t{30}; 
\def \c{5}; 
\begin{scope}[rotate=180]
    \node at (0,0) {*};
    \foreach \d in {\r,\R}{  
        \draw (\t:\d cm) arc (\t:180-\t:\d cm);
        \draw (180+\t:\d cm) arc (180+\t:360-\t:\d cm);
    }
    \draw (\t:\r cm) ..controls (\c:\r cm) and (\c:\R cm) ..  (\t:\R cm);
    \draw (-\t:\r cm) ..controls (-\c:\r cm) and (-\c:\R cm) ..  (-\t:\R cm);
    \begin{scope}[xscale=-1]
        \draw (\t:\r cm) ..controls (\c:\r cm) and (\c:\R cm) ..  (\t:\R cm);
    \draw (-\t:\r cm) ..controls (-\c:\r cm) and (-\c:\R cm) ..  (-\t:\R cm);
    \end{scope}
    \draw[ultra thick, red] (2.3*\c:-\m cm) -- (-2.3*\c:-\m cm);
\end{scope}
\node at (3, 0) {$\longrightarrow$};
\begin{scope}[xshift=6cm, rotate=180]
    \node at (0,0) {*};
    \foreach \d in {\r,\R}{  
        \draw (\t:\d cm) arc (\t:360-\t:\d cm);
    }
    \draw (\t:\r cm) ..controls (\c:\r cm) and (\c:\R cm) ..  (\t:\R cm);
    \draw (-\t:\r cm) ..controls (-\c:\r cm) and (-\c:\R cm) ..  (-\t:\R cm);
    \draw[ultra thick, blue] (-\r,0) -- (-\R,0);
\end{scope}
\end{tikzpicture}
        & \begin{tikzpicture}[scale=.5]
\def \r{1};
\def \R{2};
\def \m{1.5}; 
\def \t{30}; 
\def \c{5}; 
\begin{scope}[rotate=180]
        \node at (0,0) {*};
        \foreach \d in {\r,\R}{  
            \draw (\t:\d cm) arc (\t:360-\t:\d cm);
        }
        \draw (\t:\r cm) ..controls (\c:\r cm) and (\c:\R cm) ..  (\t:\R cm);
        \draw (-\t:\r cm) ..controls (-\c:\r cm) and (-\c:\R cm) ..  (-\t:\R cm);
        \draw[ultra thick, red] (-\r,0) -- (-\R,0);
\end{scope}
\node at (3, 0) {$\longrightarrow$};
\begin{scope}[xshift=6cm, rotate=180]
    \node at (0,0) {*};
    \foreach \d in {\r,\R}{  
        \draw (\t:\d cm) arc (\t:180-\t:\d cm);
        \draw (180+\t:\d cm) arc (180+\t:360-\t:\d cm);
    }
    \draw (\t:\r cm) ..controls (\c:\r cm) and (\c:\R cm) ..  (\t:\R cm);
    \draw (-\t:\r cm) ..controls (-\c:\r cm) and (-\c:\R cm) ..  (-\t:\R cm);
    \begin{scope}[xscale=-1]
        \draw (\t:\r cm) ..controls (\c:\r cm) and (\c:\R cm) ..  (\t:\R cm);
    \draw (-\t:\r cm) ..controls (-\c:\r cm) and (-\c:\R cm) ..  (-\t:\R cm);
    \end{scope}
    \draw[ultra thick, blue] (2.3*\c:-\m cm) -- (-2.3*\c:-\m cm);
\end{scope}
\end{tikzpicture} 
        \\ \hline
        \begin{tikzpicture}[scale=.5]
\def \r{1};
\def \R{2};
\def \m{1.5}; 
\def \t{30}; 
\def \c{5}; 
\begin{scope}
    \node at (-\m,0) {*};
    \draw (0,0) circle (\r cm);
    \draw (0,0) circle (\R cm);
    \draw[ultra thick, red] (\r,0) -- (\R,0);
\end{scope}
\node at (3, 0) {$\longrightarrow$};
\begin{scope}[xshift=6cm]
    \node at (-\m,0) {*};
    \foreach \d in {\r,\R}{  
        \draw (\t:\d cm) arc (\t:360-\t:\d cm);
    }
    \draw (\t:\r cm) ..controls (\c:\r cm) and (\c:\R cm) ..  (\t:\R cm);
    \draw (-\t:\r cm) ..controls (-\c:\r cm) and (-\c:\R cm) ..  (-\t:\R cm);
    \draw[ultra thick, blue] (2.3*\c:\m cm) -- (-2.3*\c:\m cm);
\end{scope}
\end{tikzpicture}
        & \begin{tikzpicture}[scale=.5]
\def \r{1};
\def \R{2};
\def \m{1.5}; 
\def \t{30}; 
\def \c{5}; 
\begin{scope}[xshift=10cm]
    \begin{scope}
        \node at (-\m,0) {*};
        \foreach \d in {\r,\R}{  
            \draw (\t:\d cm) arc (\t:360-\t:\d cm);
        }
        \draw (\t:\r cm) ..controls (\c:\r cm) and (\c:\R cm) ..  (\t:\R cm);
        \draw (-\t:\r cm) ..controls (-\c:\r cm) and (-\c:\R cm) ..  (-\t:\R cm);
        \draw[ultra thick, red] (2.3*\c:\m cm) -- (-2.3*\c:\m cm);
    \end{scope}
    \node at (3, 0) {$\longrightarrow$};
    \begin{scope}[xshift=6cm]
        \node at (-\m,0) {*};
        \draw (0,0) circle (\r cm);
        \draw (0,0) circle (\R cm);
        \draw[ultra thick, blue] (\r,0) -- (\R,0);
    \end{scope}
\end{scope}
\end{tikzpicture} 
        \\ \hline
        \begin{tikzpicture}[scale=.5]
\def \r{1};
\def \R{2};
\def \m{1.5}; 
\def \t{30}; 
\def \c{5}; 
\begin{scope}
    \node at (0,0) {*};
    \draw (0,0) circle (\r cm);
    \draw (0,0) circle (\R cm);
    \draw[ultra thick, red] (\r,0) -- (\R,0);
\end{scope}
\node at (3, 0) {$\longrightarrow$};
\begin{scope}[xshift=6cm]
    \node at (0,0) {*};
    \foreach \d in {\r,\R}{  
        \draw (\t:\d cm) arc (\t:360-\t:\d cm);
    }
    \draw (\t:\r cm) ..controls (\c:\r cm) and (\c:\R cm) ..  (\t:\R cm);
    \draw (-\t:\r cm) ..controls (-\c:\r cm) and (-\c:\R cm) ..  (-\t:\R cm);
    \draw[ultra thick, blue] (2.3*\c:\m cm) -- (-2.3*\c:\m cm);
\end{scope}
\end{tikzpicture}
        & \begin{tikzpicture}[scale=.5]
\def \r{1};
\def \R{2};
\def \m{1.5}; 
\def \t{30}; 
\def \c{5}; 
\begin{scope}[xshift=10cm]
    \begin{scope}
        \node at (0,0) {*};
        \foreach \d in {\r,\R}{  
            \draw (\t:\d cm) arc (\t:360-\t:\d cm);
        }
        \draw (\t:\r cm) ..controls (\c:\r cm) and (\c:\R cm) ..  (\t:\R cm);
        \draw (-\t:\r cm) ..controls (-\c:\r cm) and (-\c:\R cm) ..  (-\t:\R cm);
        \draw[ultra thick, red] (2.3*\c:\m cm) -- (-2.3*\c:\m cm);
    \end{scope}
    \node at (3, 0) {$\longrightarrow$};
    \begin{scope}[xshift=6cm]
        \node at (0,0) {*};
        \draw (0,0) circle (\r cm);
        \draw (0,0) circle (\R cm);
        \draw[ultra thick, blue] (\r,0) -- (\R,0);
    \end{scope}
\end{scope}
\end{tikzpicture} 
        \\ \hline
        \end{tabular}
        \caption{Left column, from top to bottom: merges $\wmerge$, $\mixmerge$, and $\veemerge$, respectively. Right column, from top to bottom: splits $\wsplit$, $\mixsplit$, and $\veesplit$.}
        \label{fig:akh-diffs}
    \end{figure}

    Each merge or split is a cobordism between the collections of planar circles. For instance, a merge of two circles into one can be locally represented as a ``pair of pants'' cobordism, shown in Figure \ref{fig:pants-cobordism}. 
    Throughout, we use {\color{red} red arcs on 0-resolutions} to indicate the surgery arc, so that surgering along this red arc results in the 1-resolution of the crossing. 
    Dually, we use {\color{blue} blue arcs on 1-resolutions} to indicate the dual surgery arc. 
    
    \begin{figure}
        \centering
        \labellist
            \pinlabel {{\color{red} $0$}} at 120 230
            \pinlabel {{\color{blue} $1$}} at 145 35
        \endlabellist
        \includegraphics[width=1.5in]{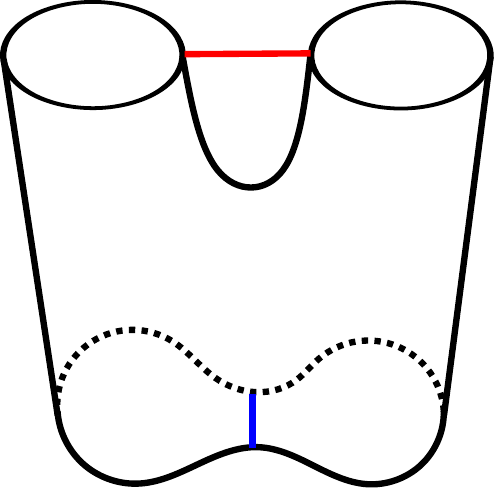}
        \caption{A pair of pants cobordism, which represents a merging of two circles when read from top to bottom. The red arc is the surgery arc on the 0-resolution; the blue arc is the dual surgery arc on the 1-resolution.}
        \label{fig:pants-cobordism}
    \end{figure}

    \begin{remark}
        Throughout, we use the term `arc' to refer to `surgery arc', as explained above. 
        Unfortunately, components of knot diagrams between crossings are also sometimes referred to as arcs; we will not be using this terminology.
    \end{remark}

\subsection{The annular Khovanov chain complex} \label{subsec:akhchaincomplex}
    
    With the combinatorics of crossing resolution switches handled, we are prepared to transform the cube of resolutions from a graph consisting of topological data to a chain complex consisting of algebraic data.
    
    Let $D_u$ be a given complete resolution corresponding to some binary string $u$. Begin by numbering each circle in $D_u$.

    Let $\vv = \Z v_+ \oplus \Z v_-$ and $\ww = \Z w_+ \oplus \Z w_-$; these are bi-graded modules with   quantum grading $\gr_q$ and $k$-grading $\gr_k$ determined by the bi-gradings on their generators:
    \begin{itemize}
        \item $(\gr_q, \gr_k) (v_\pm) = (\pm 1, \pm 1)$
        \item $(\gr_q, \gr_k) (w_\pm) = (\pm 1, 0)$.
    \end{itemize}

    \begin{remark}
    In terms of the $\essl_2(\C)$ action, $\vv \otimes_\Z \C$ is the fundamental representation for $\essl_2(\C)$, and $\ww \otimes_\Z \C$ is two copies of the trivial representation. 
    The annular $k$-grading is the $\essl_2(\C)$ weight-space grading.
    \end{remark}
    
    Define the module $\CAKh(D_u)$ as 
        \[
            \CAKh(D_u) := \bigotimes_{i} 
            \begin{cases}
            \vv & \text{ if the $i$th circle is nontrivial}\\
            \ww & \text{ if the $i$th circle is trivial}
            \end{cases}.
        \]
    Up to reordering, we may number our circles such that the module $\CAKh(D_u)$ can be written as    
    \[
        \vv^{\otimes \text{\# of nontrivial circles in $D_u$}} \otimes \ww^{\otimes \text{\# of trivial circles in $D_u$}}.
    \]

    The pure tensors that generate $\CAKh(D_u)$ are called \emph{distinguished generators}. 
    Observe that these are homogeneous in bidegree, and are easily computed by the formulas
    
    \begin{align*}
        \gr_q &= |u| + \# v_+ + \# w_+ - \# v_- - \# w_- \\
        \gr_k &= \#v_+ - \# v_-.
    \end{align*}
    
    \begin{remark}
        We can interpret a distinguished generator $x$ in a summand $\CAKh(D_u)$ as a choice of orientation for each trivial and nontrivial circle in $D_u$, with $v_+, w_+$ associated with a counterclockwise (CCW) orientation, and $v_-, w_-$ associated with a clockwise (CW) orientation. 
        Denote this oriented 1-manifold embedded in the annulus by $(D_u,x)$. 
        
        The reader may check that the $\gr_k$ of distingiushed generator corresponds to the \emph{winding number} of the oriented 1-manifold $(D_u, x)$  around the $*$: this is the algebraic intersection number of $(D_u, x)$ and a generic curve $\gamma$ connecting $*$ with $\infty$. 
    \end{remark}

    The next step is to associate to every edge a module map. Along an edge $D_u \to D_v$, a circle in $D_u$ or $D_v$ is called \emph{active} if it abuts the surgery arc or the dual surgery arc. Otherwise, it is \emph{passive}. 
    Below, we give the explicit definition of the module map $d_{u,v}$ associated to the cobordism $D_u \to D_v$. 

    \begin{definition}
    \label{defn:akh-diffs}
    The annular Khovanov differentials are explicitly defined as follows, extended by identity on passive circles. (Refer to Figure \ref{fig:akh-diffs} for the associated cobordisms.)
    \begin{itemize}
        \item A merge of two trivial circles $\wmerge$:
            \label{diff:wmerge}
            \begin{align*}
                d_{u, v} : \ww \otimes \ww & \rightarrow \ww \\
                w_+ \otimes w_+ & \mapsto w_+ \\
                w_- \otimes w_+ & \mapsto w_- \\
                w_+ \otimes w_- & \mapsto w_- \\
                w_- \otimes w_- & \mapsto 0 
            \end{align*}
        \item A merge of a nontrivial circle and a trivial circle $\mixmerge$:
            \label{diff:mixmerge}
            \begin{align*}
                d_{u,v} : \vv \otimes \ww & \rightarrow \vv \\
                v_+ \otimes w_+ & \mapsto v_+ \\
                v_+ \otimes w_- & \mapsto 0 \\
                v_- \otimes w_+ & \mapsto v_- \\
                v_- \otimes w_- & \mapsto 0 
            \end{align*}
        \item A merge of two nontrivial circles $\veemerge$:
            \label{diff:veemerge}
            \begin{align*}
                d_{u,v}: \vv \otimes \vv & \rightarrow \ww \\
                v_+ \otimes v_+ & \mapsto 0 \\
                v_- \otimes v_+ & \mapsto w_- \\
                v_+ \otimes v_- & \mapsto w_- \\
                v_- \otimes v_- & \mapsto 0 
            \end{align*}
        \item A split of one trivial circle into two trivial circles $\wsplit$:
            \label{diff:wsplit}
            \begin{align*}
                d_{u,v}: \ww & \rightarrow \ww \otimes \ww \\
                w_+ &\mapsto w_- \otimes w_+ + w_+ \otimes w_- \\
                w_- &\mapsto w_- \otimes w_-
            \end{align*}
        \item A split of one nontrivial circle into a nontrivial circle and a trivial circle $\mixsplit$:
            \label{diff:mixsplit}
            \begin{align*}
                d_{u,v}: \vv & \rightarrow \vv \otimes \ww \\
                v_+ & \mapsto v_+ \otimes w_- \\
                v_- & \mapsto v_- \otimes w_- 
            \end{align*}
        \item A split of one trivial circle into two nontrivial circles $\veesplit$:
            \label{diff:veesplit}
            \begin{align*}
                d_{u,v}: \ww & \rightarrow \vv \otimes \vv \\
                w_+ & \mapsto v_- \otimes v_+ + v_+ \otimes v_- \\
                w_- & \mapsto 0 
            \end{align*}
    \end{itemize}
    \end{definition}

    The reader may wish to verify that these maps preserve the $(\gr_q, \gr_k)$ bi-grading.
    
    \begin{notation}
    \label{notation:dots-and-arrows}
    We will also take a more combinatorial perspective on these modules and maps. 
    Any chain $\bold{x} \in \CAKh(D)$ can be written as a finite sum of distinguished generators $\sum_{i} c_i g_i$.
    Let $\langle x, g_i \rangle$ denote the coefficient of the distinguished generator $g_i$ in $\bold{x}$, using the distinguished basis. 

    Let $g, g'$ be distinguished generators, 
    and observe that 
    \[
        \langle d(g), g' \rangle 
        \in \{-1, 0, 1\}.
    \]
    This special property of Khovanov differentials allows us to interpret the chain complex $(\CAKh(D), d)$ using ``dots and arrows'', as follows and as depicted in Figure \ref{fig:dots-and-arrows}.
    
    \begin{itemize}
        \item Each distinguished generator is represented by a dot.
        \item The differential is a collection of arrows; there is an arrow from $g$ to $g'$ if and only if $\langle d(g), g' \rangle = \pm 1$.
        (We may decorate the arrow if we wish to keep track of signs.)
    \end{itemize}
    
    Suppose there is an arrow $a$ from $g$ to $g'$. 
    Then we call $g$ the \emph{source} of the arrow $a$, 
    and $g'$ the \emph{target} of the arrow $a$. 
    \end{notation}

    \begin{figure}
        \centering
        \begin{tikzpicture}[xscale=.8, yscale=1.2]
    \node[label=below:{$v_+ \otimes v_+$}] (++) at (0,2) {$\bullet$};
    \node[label=below:{$v_+ \otimes v_-$}] (+-) at (-1,1) {$\bullet$};
    \node[label=below:{$v_- \otimes v_+$}] (-+) at (1,1) {$\bullet$};
    \node[label=below:{$v_- \otimes v_-$}] (--) at (0,0) {$\bullet$};
    \node[label=below:{$w_+$}] (+) at (4,2) {$\bullet$};
    \node[label=below:{$w_-$}] (-) at (4,1) {$\bullet$};
    \draw[->] (+-) to[out=10, in=155] node[near end, above]{$a$} (-);
    \draw[->] (-+) to node[midway, below]{$b$} (-);
\end{tikzpicture}
        \caption{A dots-and-arrows depiction of a $\Z$-module merge map $m: \vv \otimes \vv \to \ww$. There are two arrows $a,b$ with sources $v_+ \otimes v_-$ and $v_- \otimes v_+$, respectively. The target of both arrows is $w_-$.}
        \label{fig:dots-and-arrows}
    \end{figure}
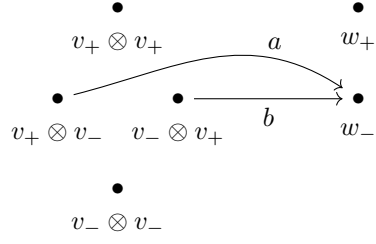

    We now have the data of a collection of modules, each on the vertices of our graph $\Gamma$, along with module maps corresponding to the edges in $\Gamma$. 

    To combine these into a chain complex, we first assign signs $\epsilon_{u,v} \in \{\pm 1\}$ to each edge $u \prec_1 v$ in the cube, so that each square in $\Gamma$ has an odd number of $-1$'s. 
    Then, we define the  (unshifted; see Remark \ref{rem:grading-shifts}) \emph{annular Khovanov chain complex} at (co)homological\footnote{Khovanov homology was defined with cohomological grading conventions; the differential increases homological grading by 1.} grading $i$ to be  
    \[
        \CAKh^i(D) = \bigoplus_{|u| = i} \CAKh(D_u)
    \]
    and the differential $d_i: \CAKh^i(D) \to \CAKh^{i+1}(D)$ by
    \[
        d_i = \bigoplus_{|u| = i} \sum_{u \prec_1 v} \epsilon_{u,v} d_{u,v}.
    \]
    To check that $d^2 = 0$, it suffices to check that any face of the cube 
    
    \begin{center}
    \begin{tikzcd}
        & v \arrow{dr}{d_{v,w}}& \\
        u \arrow{ur}{d_{u,v}} \arrow{dr}[swap]{d_{u,v'}} & & w \\
        & v' \arrow{ur}[swap]{d_{v',w}} & 
    \end{tikzcd}
    \end{center}
    commutes, and that the additional signs make the faces anti-commute. 

    The homology of this complex is denoted $\AKh(D)$. 
    Let $\AKh(L; k) = \AKh(D;k)$ denote the annular Khovanov homology of $L$ at $\gr_k = k$. 

    \begin{remark}
        \label{rem:grading-shifts}
        The expert reader may notice that our definition of $\AKh$ is not quite the usual definition, because we omit the global homological and quantum shifts that make annular Khovanov homology diagram-independent.
        The reason for this is that we are focused on the $k$-grading, which does not have a global shift, and it is easier to work directly with the unshifted complex.

        For a diagram $D$ of an annular link $L$, with $n_\pm$ crossings with sign $\pm$, respectively, 
        our unshifted $\AKh(D)$ is related to the usual annular Khovanov homology $AKh(L)$, as it appears in the literature, by 
        \[
            AKh^{i,j,k}(L) = \AKh^{i-n_-, j+n_+-2n_, k}(D).
        \]

        In particular, if we forget the homological and quantum gradings, we are justified in writing $\AKh(L; k)$, as this is indeed independent of the choice of diagram $D$.
    \end{remark}

\subsection{Wrapping number of diagrams and resolutions} \label{subsec:wrappingnumberinfo}

    Let $D$ be a diagram for an annular link $L$. 
    We define the \textit{wrapping number of $D$}, denoted $\wrap(D)$, to be the minimal geometric intersection between $D$ and a meridional arc in $A$. 
    Then the \textit{wrapping number} of $L$ is
        \[\wrap(L) = \min_{\mathrm{diagrams~} D} \wrap(D).\]
    For a resolution $D_u$ of $D$, we can similarly define $\wrap(D_u)$; $\wrap(D_u)$ can also be computed as the number of nontrivial circles in $D_u$.

    If we suppose that $D$ satisfies $\wrap(L) = \wrap(D)$, then every resolution of $D$ will have wrapping number less than or equal to $\wrap(L)$. 
    Consequently, computing the annular Khovanov homology from our cube of resolutions for $D$ will produce generators $x_i$ with $|\grk(x_i)| \leq \wrap(L)$.  
    Therefore $\AKh(L)$ is nontrivial only within these $k$-gradings as well.

    The reader may now verify that every annular Khovanov differential preserves the parity of nontrivial circles. 
    We conclude that $\AKh(L)$ must have support contained in the set of $k$-gradings
        \[\{-\wrap(L), -\wrap(L)+2, ..., \wrap(L)-2, \wrap(L)\},\]
    as \cite{GLW-sl2action} found.

\section{Perfectly wrapped and uniform resolutions} \label{sec:perfectlywrappedanduniformresolutions}

    In this section, we identify instances in which the existence of a single type of resolution suffices to verify Conjecture \ref{con:catwrapcon}, in the style of \cite{HostePrz-skeinmoduleofwhiteheadmanifolds} and \cite{Martin-wrapping}.

\subsection{Terminology and notation}

    We first establish terminology and notation for the coming sections.
    For a given resolution $D_u$, we let $|D_u| = |\pi_0(D_u)|$, i.e.\ the number of components. 
    
    We denote a cobordism merging two trivial circles into one by $\wmerge$, a cobordism merging a trivial circle a nontrivial circle into a nontrivial circle by $\mixmerge$, etc.

    For an annular link $L$, we are only concerned with the component of $\CAKh(L)$ at $\grk = \wrap(L)$ . 
    As the annular differential respects $k$-grading, we may analyze this component in isolation. 
    We initially restrict our search for nonvanishing classes to resolutions with wrapping number equivalent to $\wrap(L)$. 

    \begin{definition}
        \label{def:exactlywrapped}
        Let $D$ be a diagram for an annular link $L$. 
        We say that a resolution $D_u$ of $D$ is \textit{exactly wrapped} if the number of nontrivial circles in $D_u$ is equal to $\wrap(L)$.  
    \end{definition}

    To justify our frequent usage of these resolutions, we prove that they exist for arbitrary diagrams. 
    We will show a couple of lemmas to aid in this proof.

    \begin{lemma}
        \label{lma:paritymatchwrapdwrapl}
        For any annular diagram $D$ for an annular link $L$, $\wrap(D) \equiv \wrap(L) \mod 2$.
    \end{lemma}

    \begin{proof}
        The diagram $D \looparrowright A$ represents some $\Z/2$ homology class $[D] \in H_1(A; \Z/2) \cong \Z/2$.
        This is invariant under annular isotopy and annular Reidemeister moves, 
        so $\wrap(D) \equiv \wrap(D') \mod 2$ for any diagrams $D, D'$ of $L$.
        In particular, $\wrap(D) \equiv \wrap(L) \mod 2$, because $\wrap(L) \mod 2 = [L] \in H_1(A \times I; \Z/2) \cong \Z/2$.
    \end{proof}

    \begin{lemma}
        \label{lma:resolutionsbetweenwrapdwrapdmod2}
        Given any connected diagram $D$ for $L$, there are resolutions $D_u$ for each wrapping number between $\wrap(D) \mod 2$ and $\wrap(D)$ (of the correct parity).
    \end{lemma}
    
    \begin{proof}
        Take a resolution $D_u$ of $D$ such that $\wrap(D_u) = \wrap(D)$. 
        In particular, $\wrap(D_u)$ is maximal out of all resolutions for $D$.
        As $D$ is connected, we may use a series of cobordisms of the type $\wmerge$ and $\mixmerge$ to obtain a resolution $D_v$ of $D$ where $\wrap(D_v) = \wrap(D)$ and every circle in $D_v$ is nontrivial. 
        Using the fact that $D$ is connected once again, each nontrivial circle in $D_v$ must abut some other nontrivial circle in $D_v$ by a crossing resolution.
        We can then obtain the desired resolutions by repeated applications of merges $\veemerge$, which preserve the parity of the number of nontrivial circles in a resolution.
    \end{proof}

    Let $D, D'$ be annular link diagrams, viewed as diagrams in $S^1 \times [1,2]$ and $S^1 \times [2,3]$, respectively. 
    The \emph{horizontal composition} of $D$ and $D'$ is the annular diagram $D \sqcup D'$ in $S^1 \times [1,3]$. 
    
    \begin{proposition}
        \label{prop:exactlywrappedresexists}
        Let $D$ be a diagram for an annular link $L$. 
        Then $D$ has an exactly wrapped resolution.
    \end{proposition}

    \begin{proof}
        From Lemma \ref{lma:paritymatchwrapdwrapl}, we know that $\wrap(D) \equiv \wrap(L) \mod 2$. 
        Moreover, $\wrap(D) \geq \wrap(L)$ by definition. 
        Lastly, for a diagram $D = D_1 \sqcup ... \sqcup D_n$ consisting of the horizontal composition of $n$ connected diagrams, we have 
            \[\wrap(D) = \wrap(D_1) + ... + \wrap(D_n).\]
        Our proposition follows immediately from Lemma \ref{lma:resolutionsbetweenwrapdwrapdmod2} applied to $D_1, ..., D_n$. 
    \end{proof}

    We will invoke Proposition \ref{prop:exactlywrappedresexists} implicitly throughout the coming sections.

    Within a collection of resolutions, we may select especially useful ones with restrictions on their incoming and outgoing cobordisms. 
    
    \begin{definition}
        \label{def:insulated}
        Let $D$ be a diagram for an annular link $L$. 
        A resolution $D_u$ is \textit{insulated} if 
        \begin{itemize}
            \item Every cobordism with source $D_u$ is of the form $\wmerge, \mixmerge, \veemerge,$ or $\veesplit$.
            \item Every cobordism with target $D_u$ is of the form $\wsplit, \mixsplit, \veesplit,$ or $\veemerge$. 
        \end{itemize}
    \end{definition}

    In other words, the resolutions adjacent to $D_u$ either have strictly less trivial circles than $D_u$, or have a different number of nontrivial circles. 
    This definition is similar to the condition on $0$- and $1$-smoothings of links in the classical definition of adequacy.

    \begin{definition}
        \label{def:perfectlywrapped}
        A resolution $D_u$ is \textit{perfectly wrapped} if $D_u$ is both exactly wrapped and insulated. 
    \end{definition}

    We now show that perfectly wrapped resolutions exist as a corollary of Proposition \ref{prop:exactlywrappedresexists}.

    \begin{lemma}
        \label{lma:existenceofperfectlywrappedres}
        Every diagram $D$ for an annular link $L$ admits a perfectly wrapped resolution.
    \end{lemma}

    \begin{proof}
        Let $D_u$ be an exactly wrapped resolution of $D$ such that, for any other exactly wrapped resolution $D_v$,
            \[\text{\# of trivial circles in $D_u$} \geq \text{\# of trivial circles in $D_v$}.\]
        Suppose that $D_v$ is an exactly wrapped resolution for which there is a cobordism $D_v \rightarrow D_u$. 
        Then this cobordism must be of the form $\wsplit$ or $\mixsplit$ by the above condition. 
        If $D_w$ is an exactly wrapped resolution for which there is a cobordism $D_u \rightarrow D_w$, then this cobordism must be of the form $\wmerge$ or $\mixmerge$ in a similar fashion.
        This ensures that every incoming and outgoing cobordism is as desired.
    \end{proof}

    From the perspective of a single resolution, we wish to analyze the data of its incident cobordisms. 
    
    \begin{definition}
        \label{def:keratimgat}
        We define the \textit{null set of $D_u$}, written $\kerat(D_u)$, to be the set of all distinguished generators which are not the source of any arrow at $\grk = \wrap(L)$.
        We define the \textit{target set of $D_u$}, written $\imgat(D_u)$, to be the set of all distinguished generators that are targets for some arrow at $\grk = \wrap(L)$.
    \end{definition}
    
    Usually, $\imgat(D_u)$ is not a subset of $\kerat(D_u)$ in absence of a cocycle condition. 
    However, these objects still may allow us to verify Conjecture \ref{con:catwrapcon} in the case that $\kerat(D_u) \setminus \imgat(D_u) \neq \emptyset$, as a distinguished generator $x \in \kerat(D_u) \setminus \imgat(D_u)$ represents a nonvanishing homology class in $\AKh$. We can also obtain information about $D_u$.

    \begin{lemma}
        \label{lma:localtoglobalakh}
        If $\kerat(D_u) \setminus \imgat(D_u) \neq \emptyset$, then $D_u$ is insulated. Furthermore,  $\AKh(L; \wrap(L)) \ncong 0$. 
    \end{lemma}

    \begin{proof}
        For the first statement, we will prove the contrapositive. 
        A cobordism with source $D_u$ of the form $\wsplit$ or $\mixsplit$ induces an always nonzero map in $\wrap(L)$ grading, implying that $\kerat(D_u) = \emptyset$. 
        A cobordism with target $D_u$ of the form $\wmerge$ or $\mixmerge$ induces a map whose image spans $\CAKh(D_u)$, meaning that $\kerat(D_u) \subseteq \imgat(D_u)$, and our claim holds.

        For the second statement, let $x \in \kerat(D_u) \setminus \imgat(D_u)$. 
        Then the element  
            \[(0,...,0, x, 0, ..., 0) \in \CAKh^{|u|}(D)\]
        is in the kernel of $d_{|u|}$, as $x \in \kerat(D_u).$
        As $x \notin \imgat(D_u)$, $x$ is not the target of a single arrow by definition.
        Consequently, $d_{|u|-1}$ does not contain $(0,...,0, x, 0, ..., 0)$ in its image.
        Therefore $\AKh(L; \wrap(L)) \ncong 0$.
    \end{proof}

    Our goal is to discern exactly which resolutions satisfy the conditions of Lemma \ref{lma:localtoglobalakh}, so that we may verify Conjecture \ref{con:catwrapcon} for the link we are resolving.

\subsection{Type $0$ and type $1$ trivial circles} \label{subsec:type01circles}

    We now define the qualities of trivial circles which we seek to investigate.
    
    \begin{definition}
        \label{def:type01}
        Let $D_u$ be a resolution. We say a trivial circle $W$ in $D_u$ is \textit{type $0$} if it only abuts $0$-resolutions of crossings (i.e.\ red arcs).
        Similarly, we say a trivial circle $W$ in $D_u$ is \textit{type $1$} if it 
        only abuts $1$-resolutions of crossings (i.e.\ blue arcs).
    \end{definition}

    In the lemmas below, we use the shorthand notation $w_\pm$ to mean `either $w_+$ or $w_-$'. 

    \begin{lemma}
        \label{lma:kerclassif}
        Let $D_u$ be perfectly wrapped. 
        Let $n_1$ denote the number of type $1$ trivial circles in $D_u$. 
        Then, for a distinguished generator $x \in \CAKh(D_u)$, we have 
            \[x \in \kerat(D_u)\]
        if and only if $x$ is of the form 
            \begin{equation}
            \label{eq:not1kernel}
                \underbrace{v_+^{\otimes \wrap(L)}}_{\text{nontrivial}} 
                \otimes 
                \underbrace{w_\pm^{\otimes n_1}}_{\text{type $1$}}
                \otimes 
                \underbrace{w_-^{\otimes (|D_u| - \wrap(L) - n_1)}}_{\text{not type $1$}},
            \end{equation}            
        where every trivial circle that is \textbf{not} type $1$ is labeled $w_-$.
    \end{lemma}

    \begin{proof}
        Recall that $\kerat(D_u)$ only contains distinguished generators with $k$-grading equal to $\wrap(L)$, by definition. 
        Therefore, for a distinguished generator to be in $\kerat(D_u)$, every nontrivial circle must be labeled $v_+$. 
    
        For the matter of trivial circles, observe that a trivial circle $W$ in $D_u$ is active for some cobordism with source $D_u$ if and only if $W$ is not type $1$.

        We show that if a distinguished generator is not of the form above, then it is not in the kernel.
        Consider a distinguished generator     
            \begin{equation*}
            \label{eq:outsideofkernel}
                y = 
                \underbrace{v_+^{\otimes \wrap(L)}}_{\text{nontrivial}} 
                \otimes 
                \underbrace{w_\pm^{\otimes n_1}}_{\text{type $1$}}
                \otimes 
                \underbrace{w_+}_{\text{not type $1$}}
                \otimes 
                \underbrace{w_\pm^{\otimes (|D_u| - \wrap(L) - n_1 - 1)}}_{\text{not type $1$}},
            \end{equation*}     
        where some trivial circle $W$ that is not type $1$ is labeled $w_+$. 
        As $W$ is not type $1$, $W$ is active for some cobordism $D_u \rightarrow D_v$. 
        The map induced by this cobordism will be either 
        $\ww \otimes \ww \rightarrow \ww$, $\vv \otimes \ww \rightarrow \ww$, or $\ww \rightarrow \vv \otimes \vv$.
        By inspection of the $\AKh$ differentials (Definition \ref{defn:akh-diffs}), one may check that $y$ is not in the kernel of any of these maps. 
        Hence, $y \notin \kerat(D_u)$.
        
        On the other hand, let $x$ be a distinguished generator of the form shown in \eqref{eq:not1kernel}, where every trivial circle that is not type $1$ is labeled $w_-$. 
        As $D_u$ is perfectly wrapped, there are four possible cobordisms with source $D_u$, namely $\wmerge$, $\mixmerge$, $\veemerge$, and $\veesplit$.

        \begin{enumerate}
            \item Consider the map $m_{W, W'}: \ww \otimes \ww \rightarrow \ww$ corresponding to the cobordism $\wmerge$, which merges trivial circles $W$ and $W'$.
            As both trivial circles are active for this cobordism, they are connected by a $0$ crossing resolution, meaning that neither are type $1$. 
            Therefore, they are both labeled $w_-$. Hence, $m_{W, W'}(x) = 0$.
            \item For the map $m_{V, W} : \vv \otimes \ww \rightarrow \vv$ corresponding to the cobordism $\mixmerge$, merging a nontrivial circle $V$ and a trivial circle $W$, we know that $W$ is not type $1$ as it is active for this cobordism. 
            Hence, it is labeled $w_-$, implying $m_{V, W}(x) = 0$.
            \item Next, we consider the map $m_{V,V'}: \vv \otimes \vv \rightarrow \ww$ corresponding to the cobordism $\veemerge$, which merges nontrivial circles $V$ and $V'$. 
            This map is trivial at the $\wrap(L)$ $k$-grading, so $m_{V,V'}(x) = 0$.
            \item Finally, we have the map $\Delta_{V, V'}: \ww \rightarrow \vv \otimes \vv$ corresponding to the cobordism $\veesplit$, which splits a trivial circle $W$ into two nontrivial circles $V, V'$. 
            Given that $W$ is active for this cobordism, $W$ is not type $1$, and is thus labeled by $w_-$. 
            Therefore, $\Delta_{V, V'}(x) = 0$. 
        \end{enumerate}
        Therefore, along each edge of the cube of resolutions with source $D_u$, the induced map sends $x$ to $0$, so $x \in \kerat(D_u)$.
    \end{proof}

    The following lemma is, in a sense, dual to the previous lemma.

    \begin{lemma}
        \label{lma:imgclassif}
        Let $D_u$ be perfectly wrapped. 
        Let $n_0$ denote the number of type $0$ trivial circles in $D_u$.
        Then, for a distinguished generator $x \in \CAKh(D_u)$, we have 
            \[x \notin \imgat(D_u)\]
        if and only if $x$ is of the form        
            \begin{equation}
            \label{eq:not0image}
                \underbrace{v_+^{\otimes \wrap(L)}}_{\text{nontrivial}} 
                \otimes 
                \underbrace{w_\pm^{\otimes n_0}}_{\text{type $0$}}
                \otimes 
                \underbrace{w_+^{\otimes (|D_u| - \wrap(L) - n_0)}}_{\text{not type $0$}},
            \end{equation}     
        where every trivial circle that is not type $0$ is labeled $w_+$.
    \end{lemma}

    \begin{proof}
        By the same reasoning as in Lemma \ref{lma:kerclassif}, we must label our nontrivial circles by $v_+$.

        A trivial circle $W$ in $D_u$ is active for a cobordism $D_v \rightarrow D_u$ if and only if $W$ is \textbf{not} type 0. 
        
        Consider a distinguished generator           
            \begin{equation*}
            \label{eq:insideofimage}
                y = 
                \underbrace{v_+^{\otimes \wrap(L)}}_{\text{nontrivial}} 
                \otimes 
                \underbrace{w_\pm^{\otimes n_0}}_{\text{type $0$}}
                \otimes 
                \underbrace{w_-}_{\text{not type $0$}}
                \otimes
                \underbrace{w_\pm^{\otimes (|D_u| - \wrap(L) - n_0 - 1)}}_{\text{not type $0$}},
            \end{equation*}        
        where some trivial circle $W$ that is not type $0$ is labeled $w_-$. 
        By the previous remark, there is a cobordism $D_v \rightarrow D_u$ affecting $W$. 
        This cobordism will induce a map $\ww \rightarrow \ww \otimes \ww$, $\vv \otimes \vv \rightarrow \ww$, or $\vv \rightarrow \vv \otimes \ww$. 
        We may then check that $y$ is in the image of each of these maps, so that $y \in \imgat(D_u)$.

        Let $x$ be of the form shown in \eqref{eq:not0image}, where every trivial circle that is not type $0$ is labeled $w_+$. 
        As $D_u$ is perfectly wrapped, there are four possible cobordisms with target $D_u$:
        (1) $\wsplit$, 
        (2) $\mixsplit$, 
        (3) $\veesplit$, and 
        (4) $\veemerge$. 
        
        Case (3) is immediate, as the induced map is trivial in $\wrap(L)$ $k$-grading. 
        Cases (1) and (2) are both splits, where the active trivial circles in $D_u$ are all labeled $+$; therefore, we may check that $x$ is not in the image of either of these induced maps. 
        For case (4), the active trivial circle is labeled $+$. Once again, the induced map does not contain $x$ in its image.
        We conclude that $x \notin \imgat(D_u)$ as desired. 
    \end{proof}

\subsection{Proof of the main result} \label{subsec:mainresult}
    We will now introduce a key definition, one which will allow us to accomplish our goal of classifying resolutions $D_u$ for which Lemma \ref{lma:localtoglobalakh} applies.  

    \begin{definition}
        \label{def:uniform}
        We say a resolution $D_u$ of a diagram $D$ for an annular link $L$ is \textit{uniform} if every trivial circle in $D_u$ is of type $0$ or type $1$.
    \end{definition}

    Note that we allow for the possibility of split nullhomologous unknot components of $D$ in the annulus; these are trivial circles that are both type $0$ and type $1$. 

    \begin{lemma}
        \label{lma:implicationofuniformity}
        If $\kerat(D_u) \setminus \imgat(D_u) \neq \emptyset$, then $D_u$ is uniform.
    \end{lemma}

    \begin{proof}

        We will prove the contrapositive.
        Assume that $D_u$ is not uniform. 
        First note that if $D_u$ is \emph{not} insulated, then Lemma \ref{lma:localtoglobalakh} implies that $\kerat(D_u) \setminus \imgat(D_u) = \emptyset$, and we are done.
        So, we may assume that $D_u$ is insulated.

        As $D_u$ is not uniform, there exists some trivial circle $W$ in $D_u$ which abuts both a $0$-resolution and a $1$-resolution.   
        
        As $W$ abuts a $0$-resolution, it is active for a cobordism $D_u \rightarrow D_v$. 
        If $y$ is a distinguished generator in $\kerat(D_u)$, we may check by cases on the induced maps (as we do in Lemma \ref{lma:kerclassif}) that $W$ must be labeled by $w_-$ in $y$, keeping in mind that $D_u$ is insulated.
        
        As $W$ also abuts a $1$-resolution, it is active for a cobordism $D_w \rightarrow D_u$ of the form $\wsplit, \mixsplit,$ or $\veemerge$.
        Recalling that $W$ is labeled $w_-$ in $y$, 
        we may again check that each of these cases places $y$ in $\imgat(D_u)$.
      
        Thus $\kerat(D_u) \subseteq \imgat(D_u)$, and so $\kerat(D_u) \setminus \imgat(D_u) = \emptyset$ as desired.
    \end{proof}

    We are now ready to prove the main theorem.

    \begin{theorem}
        \label{thm:perfectlywrappeduniform}
        Let $D$ be a diagram for an annular link $L$. 
        Then a resolution $D_u$ of $D$ is perfectly wrapped and uniform if and only if $\kerat(D_u) \setminus \imgat(D_u) \neq \emptyset$.
        In this case, let $n_0, n_1, n_2$ respectively denote the number of trivial circles which are type $0$ and not type $1$, type $1$ and not type $0$, and both type $0$ and type $1$.
        Then, up to renumbering of trivial circles, $\kerat(D_u) \setminus \imgat(D_u)$ contains exactly those distinguished generators $x$ of the form 
            \begin{equation}
            \label{eq:inkernelnotinimage} 
                v_+^{\otimes \wrap(L)}
                \otimes 
                \underbrace{w_-^{\otimes n_0}}_{\text{type $0$}}
                \otimes 
                \underbrace{w_\pm^{\otimes n_2}}_{\text{both}}
                \otimes 
                \underbrace{w_+^{\otimes n_1}}_{\text{type $1$}}.
            \end{equation}
        In particular, $L$ satisfies Conjecture \ref{con:catwrapcon}.
    \end{theorem}

    \begin{proof}  
        ($\Longrightarrow$) Suppose that $D_u$ is perfectly wrapped and uniform.    
        Up to renumbering of trivial circles, we may use Lemma \ref{lma:kerclassif} to conclude that $\kerat(D_u)$ contains exactly those distinguished generators of the form          
            \begin{equation*}
            \label{eq:not1kernelv2}
                v_+^{\otimes \wrap(L)}
                \otimes 
                \underbrace{w_-^{\otimes n_0}}_{\text{type $0$}}
                \otimes 
                \underbrace{w_\pm^{\otimes n_2} \otimes w_\pm^{\otimes n_1}}_{\text{type $1$}},
            \end{equation*}            
        in which each type $0$ trivial circle is labeled $w_-$. 
        Similarly, Lemma \ref{lma:imgclassif} tells us that $\imgat(D_u)$ exactly excludes those distinguished generators of the form
            \begin{equation*}
            \label{eq:not0imagev2}
                v_+^{\otimes \wrap(L)}
                \otimes 
                \underbrace{w_\pm^{\otimes n_0} \otimes w_\pm^{\otimes n_2}}_{\text{type $0$}}
                \otimes 
                \underbrace{w_+^{\otimes n_1}}_{\text{type $1$}},
            \end{equation*}
        in which each type $1$ trivial circle is labeled $w_+$.

        Now define
            \begin{equation}
            \label{eq:inkernelnotinimagev2}
                x = 
                v_+^{\otimes \wrap(L)}
                \otimes 
                \underbrace{w_-^{\otimes n_0}}_{\text{type $0$}}
                \otimes 
                \underbrace{w_\pm^{\otimes n_2}}_{\text{both}}
                \otimes 
                \underbrace{w_+^{\otimes n_1}}_{\text{type $1$}}.
            \end{equation}
        By inspection, $x \in \kerat(D_u) \setminus \imgat(D_u)$, so $\kerat(D_u) \setminus \imgat(D_u) \neq \emptyset$. 

        ($\Longleftarrow$) Assume that $\kerat(D_u) \setminus \imgat(D_u) \neq \emptyset$. 

        By Lemmas \ref{lma:localtoglobalakh} and \ref{lma:implicationofuniformity}, $D_u$ is insulated and uniform. Moreover, $\wrap(D_u) \geq \wrap(L)$ as $\kerat(D_u)$ is nonempty (recall that $\kerat(D_u)$ contains only generators at $k$-grading $\wrap(L)$).
        If $\wrap(D_u) > \wrap(L)$, then we may consider the distinguished generator  
            \begin{equation}
            \label{eq:inkernelnotinimagev3}
                z = 
                v_+^{\otimes \wrap(D_u)}
                \otimes 
                \underbrace{w_-^{\otimes n_0}}_{\text{type $0$}}
                \otimes 
                \underbrace{w_\pm^{\otimes n_2}}_{\text{both}}
                \otimes 
                \underbrace{w_+^{\otimes n_1}}_{\text{type $1$}}.
            \end{equation}
        For every cobordism with source $D_u$, the induced map takes $z$ to $0$. 
        Moreover, for every cobordism with target $D_u$, the induced map does not contain $z$ in its image. 
        This means that 
            \[(0, ..., 0, z, 0, ..., 0) \in \CAKh^{|u|}(D)\]
        generates a non-zero homology class in $\AKh(L; \wrap(D_u))$, contradicting the fact that $\AKh(L)$ is supported in $k$-gradings less than or equal to $\wrap(D_u)$.
        Therefore, $\wrap(D_u) = \wrap(L)$ must hold, implying that $D_u$ is perfectly wrapped and uniform by definition. 

        The final claims are fairly immediate; only distinguished generators in $\CAKh(D_u)$ of the form in \eqref{eq:inkernelnotinimagev2} are both in the null set of $D_u$ and not in the target set of $D_u$. 
        Hence, $\kerat(D_u) \setminus \imgat(D_u)$ contains exactly the generators of the form \eqref{eq:inkernelnotinimagev2}.
        An application of Lemma \ref{lma:localtoglobalakh} completes the proof, showing Conjecture \ref{con:catwrapcon} for $L$.
    \end{proof}

    \begin{corollary}
        \label{cor:lowerboundforwraplrank}
        Let $L$ be an annular link. 
        Define 
            \[U(L) = \max_{\mathrm{diagrams}~D}\#(\mathrm{perfectly~wrapped~uniform~ resolutions~of}~D).\]
        Then
            \[\rank \AKh(L; \wrap(L)) \geq U(L).\]
    \end{corollary}

    \begin{proof}
        From Theorem \ref{thm:perfectlywrappeduniform} and Lemma \ref{lma:localtoglobalakh}, every perfectly wrapped uniform resolution corresponds to at least one distinct homology class in $\AKh(L; \wrap(L))$, generating $\Z$.  
    \end{proof}

    \begin{corollary}
        \label{cor:nonexistenceofpwur}
        If $\AKh(L; \wrap(L)) / \mathrm{Tors} \cong 0$, then $U(L) = 0$. 
        Namely, $L$ does not admit a perfectly wrapped uniform resolution. 
    \end{corollary}

\subsection{An expanded class of diagrams admitting perfectly wrapped uniform resolutions} \label{subsec:expandingclassofpwur}

    We will now work to identify a general case in which a perfectly wrapped resolution can be converted into a perfectly wrapped uniform resolution; the motivation is alternating links, which we will handle in Section \ref{sec:classes}. 
    
    Recalling our description of arcs from Subsection \ref{subsec:cubeofres}, we will be specific as to the location of these arcs. 

    \begin{definition}
        \label{def:intextarc}
        Fix a trivial circle $W$ in $D_u$.
        By the Jordan curve theorem, $A \setminus W$ consists of two connected components.
        We say a crossing resolution, or equivalently, an arc, $a$ is \textit{$W$-exterior} if it is in the same connected component as the $*$, and \textit{$W$-interior} otherwise.
        Additionally, we refer to $a$ as \textit{$W$-adjacent-exterior} (resp. \textit{$W$-adjacent-interior}) if $a$ is $W$-exterior (resp. $W$-interior) and abuts $W$.
    \end{definition}
    
    A different set of desirable trivial circles can now be defined.

    \begin{definition}
        \label{def:type01type10}
        We say a trivial circle $W$ is \textit{of type $(0,1)$} if every $W$-adjacent-interior arc is red, and every $W$-adjacent-exterior arc is blue. 
        Similarly, $W$ is \textit{of type $(1,0)$} if every $W$-adjacent-interior arc is blue, and every $W$-adjacent-exterior arc is red. 
    \end{definition}

    \begin{definition}
        \label{def:almostuniform}
        A resolution $D_u$ of a diagram $D$ for an annular link $L$ is \textit{almost uniform} if every trivial circle in $D_u$ is of type $(0,1)$ or type $(1,0)$.
    \end{definition}

    Note that an almost uniform trivial circle is uniform if and only if it has no adjacent-interior arcs or no adjacent-exterior arcs. 
    
    Similar to Definition \ref{def:intextarc}, we want a notion of trivial circles on the interior and exterior of $W$. 

    \begin{definition}
        \label{def:intexttrivcircle}
        Let $W$ be a trivial circle in a resolution $D_u$.
        We say a trivial circle $W'$ is an \textit{exterior trivial circle} of $W$ if $W'$ is contained in the same connected component of $A \setminus W$ as the $*$.
        Otherwise, $W'$ is an \textit{interior trivial circle} of $W$.
    \end{definition}

    The matter of finding a perfectly wrapped uniform resolution from a perfectly wrapped almost uniform resolution comes down to analyzing interior trivial circles. 
    To make these notions precise, we establish an intuitive idea of depth.

    \begin{definition}
        \label{def:depth}
        Let $D_u$ be a resolution of a diagram $D$, and let $W \subset D_u$ be a trivial circle. 
        Let $a$ be a point just outside $W$.
        The \emph{depth} of $W$ in $D_u$, denoted by $\dep(W)$, is the minimal geometric intersection of a curve $\gamma$ connecting $a$ to $\infty$ and the set of trivial circles in $D_u$.
    \end{definition} 

    In other words, the depth of a trivial circle $W$ is the number of trivial circles containing $W$, not counting $W$.

    Finally Figure \ref{fig:nugatorycrossing} illustrates the difference between two types of nugatory crossings (for diagrams in the plane) in the presence of the marked point {\color{red}$*$}.
    The terms \textit{removable nugatory crossing} and \textit{essential nugatory crossing} come from Page 10 of \cite{HansHomayun-generalizedtait}.

    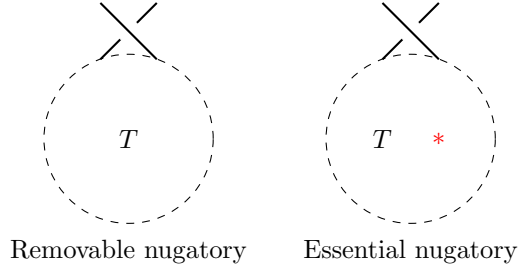
\begin{figure}
        \begin{tikzpicture}[scale=.75]
\begin{scope}[xshift=-.5cm, yshift=1.41cm]
    \draw[thick] (0,0) -- (1,1);
    \filldraw[white] (.5,.5) circle (.2cm);
    \draw[thick] (0,1) -- (1,0);
\end{scope}
\draw[dashed] (0,0) circle (1.5cm);
\node at (0,0) {$T$};
\node at (0,-2) {Removable nugatory};
\
\begin{scope}[xshift=5cm]
    \begin{scope}[xshift=-.5cm, yshift=1.41cm]
        \draw[thick] (0,0) -- (1,1);
        \filldraw[white] (.5,.5) circle (.2cm);
        \draw[thick] (0,1) -- (1,0);
    \end{scope}
    \draw[dashed] (0,0) circle (1.5cm);
    \node at (-.5,0) {$T$};
    \node at (.5,0) {{\color{red}$*$}};
    \node at (0,-2) {Essential nugatory};
\end{scope}

\end{tikzpicture}
        \caption{A removable nugatory crossing can be eliminated via Reidemeister moves representing an isotopy supported within the solid torus.}
        \label{fig:nugatorycrossing}
    \end{figure}

    \begin{theorem}
        \label{thm:expandingclass}
        Suppose that $L$ admits a diagram $D$ without removable nugatory crossings, where $D$ has a perfectly wrapped almost uniform resolution $D_u$. 
        Then $D$ admits a perfectly wrapped uniform resolution $D_v$, and in particular, $L$ satisfies Conjecture \ref{con:catwrapcon}.
    \end{theorem}

    \begin{proof}
        Let $W_1, ..., W_k$ denote the set of trivial circles in $D_u$ such that $\dep(W_i) = 0$. 
        To construct the resolution $D_v$, we will build a series of perfectly wrapped resolutions $D_{u_1}, ..., D_{u_k},$ each of which is progressively more uniform. 

        We start with $W_1 \subseteq D_u$.
        Suppose that $W_1$ is type $(0,1)$. 
        Let $D_{u_1}$ be the resolution where every $W_1$-exterior arc is unchanged and every $W_1$-interior red arc is switched to a blue arc.
        The resulting resolution is similar to $D_u$, except that $W_1$ and each of its interior trivial circles have been replaced by a collection of type $1$  circles; since the $*$ is outside $W_1$, these new circles are indeed trivial.
        Let $\ww D_{u_1}$ be the local `subresolution' consisting of this new collection of type $1$ trivial circles, along with all arcs that were $W_1$-interior in $D_u$. 

        Switching the color of a $W_1$-interior arc must induce a cobordism $\wsplit$ or $\wmerge$, since these arcs are, by definition, on the interior of a trivial circle.
        Hence, $D_{u_1}$ is exactly wrapped, as it is related to $D_u$ by cobordisms which preserve the number of nontrivial circles.
        
        We claim that $D_{u_1}$ is in fact perfectly wrapped.
        As $D_u$ is insulated, every crossing resolution in $D_u$ is of one of the forms below:

        \begin{itemize}
            \item[(I-0)] If $c$ is a $0$-resolution, then switching $c$ to a $1$-resolution induces a cobordism of the form $\wmerge, \mixmerge, \veemerge,$ or $\veesplit$.
            \item[(I-1)] If $c$ is a $1$-resolution, then switching $c$ to a $0$-resolution induces a cobordism of the form $\mathcal{W} \leftarrow \mathcal{W} \sqcup \mathcal{W}, \mathcal{V} \leftarrow \mathcal{V} \sqcup \mathcal{W}, \mathcal{W} \leftarrow \mathcal{V} \sqcup \mathcal{V},$ or $ \mathcal{V} \sqcup \mathcal{V} \leftarrow \mathcal{W}$.
        \end{itemize}
        
        If this is the case for each crossing resolution in $D_{u_1}$, then $D_{u_1}$ is insulated, and thus perfectly wrapped by definition. 
        Note that a crossing resolution that was $W_1$-exterior in $D_u$ remains the same color in $D_{u_1}$ and still satisfies either (I-0) or (I-1); 
        in particular, the $W_1$-adjacent-exterior crossing resolutions are still merge maps.
        Therefore, it suffices to show that $\ww D_{u_1}$ is insulated. 
        This will cover the case of all remaining arcs. 

        Each arc $c'$ in $\ww D_{u_1}$ is blue and abuts only trivial circles. 
        If $c'$ abuts two different trivial circles, then switching this resolution induces a cobordism $\mathcal{W} \leftarrow \mathcal{W} \sqcup \mathcal{W}$.
        Thus, for $D_{u_1}$ to be insulated, it remains to show that a given $c'$ may not abut only one trivial circle.

        Suppose for contradiction that there exists a resolution $p'$ in $D_{u_1}$, corresponding to some $W_1$-interior resolution $p$ in $D_u$, such that $p'$ abuts only one type $1$ trivial circle $X$ in $\ww D_{u_1}$.
        This phenomenon is detailed in Figure \ref{fig:not-perfectly-wrapped}.
        
        \begin{figure}
            \centering
            \labellist
                \pinlabel $X$ at -25 100
                \pinlabel $q_m'$ at 225 160
                \pinlabel $r_n'$ at 175 120
                \pinlabel $r_1'$ at 215 70
                \pinlabel $q_1'$ at 170 10
                \pinlabel $p'$ at 250 95
            \endlabellist
            \includegraphics[width=0.5\linewidth]{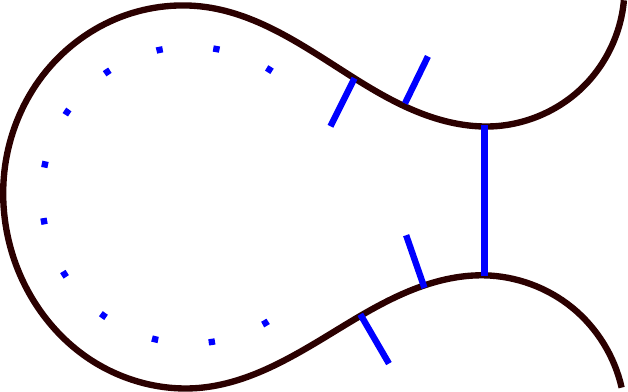}
            \caption{$p'$ is a crossing resolution in $D_{u_1}$ abutting $X$ alone.}     
            \label{fig:not-perfectly-wrapped}
        \end{figure}
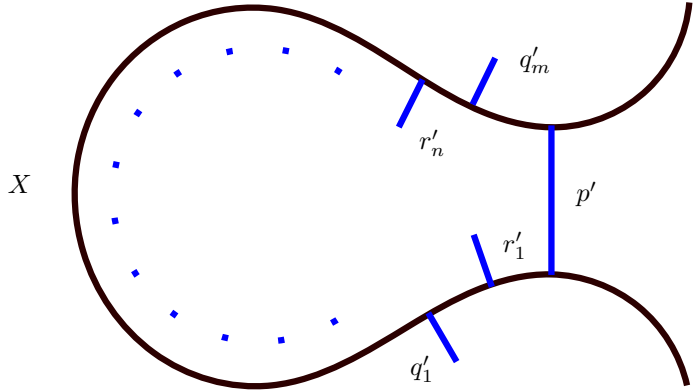
        
        For the crossing resolutions $q_i'$ and $r_i'$ in $D_{u_1}$, let $q_i$ and $r_i$ denote the respective crossing resolutions in $D_u$. 
        We now have two cases: $p$ is blue, or $p$ is red. Figure \ref{fig:case1pisblue} pictures the first case. 

        \begin{figure}
            \centering
            \labellist
                \pinlabel $p$ at 130 85
                \pinlabel $X_1$ at 117 20
                \pinlabel $X_2$ at 117 155
            \endlabellist
            \includegraphics[width=0.35\linewidth]{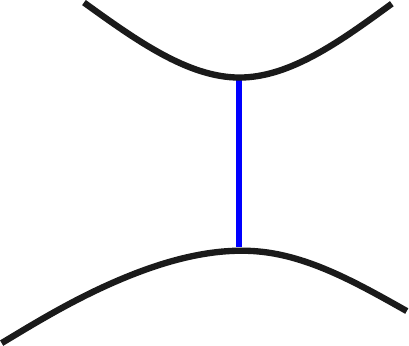}
            \caption{$p$ is a blue arc in $D_u$}
            \label{fig:case1pisblue}
        \end{figure}

        The arc $p$ is $W_1$-interior, and thus switching $p$ cannot induce a cobordism $\mathcal{V} \sqcup \mathcal{V} \leftarrow \mathcal{W}$.
        Therefore, as $D_u$ is perfectly wrapped, it must be the case that $X_1$ and $X_2$ are different trivial circles.
        
        \begin{enumerate}
            \item[(B-1)] If $m = 0$ and $n \geq 0$, then $p$ is the resolution of a removable nugatory crossing, contradicting our original assumption.
            \item[(B-2)] If $m > 0$ and $n = 0$, then $q_1$ must abut the trivial circle $X_1$.
            By almost uniformity of $D_u$, we must either have (1) $q_1$ is blue and on the same side of $X_1$ as $p$ or (2) $q_1$ is red and on the opposite side of $X_1$ as $p$. 
            In either case, we see that $q_1'$ is blue and on the same side of $X_1$ as $p$, contradicting our definition of the $q_i$ arcs. 

            \item[(B-3)] If $m > 0$ and $n > 0$, then we move clockwise around $X$ in $D_{u_1}$. 
            If $q_1$ immediately succeeds $p$ in clockwise order, then we observe the same contradiction as in Case (B-2). 
            Suppose otherwise that $r_1, ..., r_i$ succeed $p$ and precede $q_1$. If $r_i$ is blue, then we obtain the same contradiction as Case (B-2). 
            Therefore, $r_i$ is red. But then, as $D_u$ is almost uniform, $q_1$ is red as well.
            In Figure \ref{fig:secondcontradictionfromq1}, we observe that this as a contradiction.

            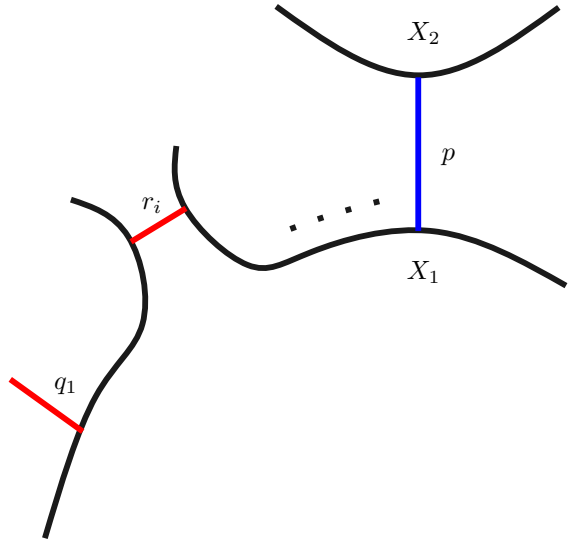
\begin{figure}
                \centering
                \labellist
                    \pinlabel $p$ at 230 200
                    \pinlabel $X_1$ at 217 140
                    \pinlabel $X_2$ at 217 265
                    \pinlabel $q_1$ at 30 80
                    \pinlabel $r_i$ at 75 175
                \endlabellist
                \includegraphics[width=0.45\linewidth]{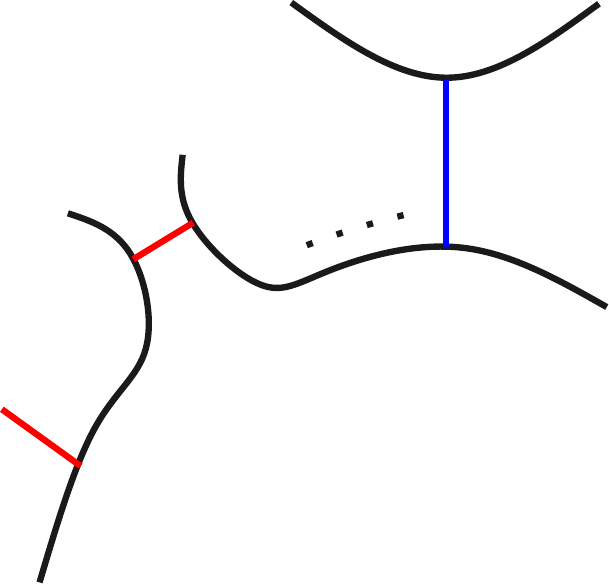}
                \caption{$q_1$ and $r_i$ are on opposite sides of a trivial circle, but are colored the same, so this trivial circle can be neither type $(0,1)$ nor type $(1,0)$}
                \label{fig:secondcontradictionfromq1}
            \end{figure}       
            
        \end{enumerate}
    
        It follows $p$ must be red, as depicted in Figure \ref{fig:case2pisred}. 
        As before, $X_1$ and $X_2$ are distinct trivial circles; we think of $X_1$ as a trivial circle that abuts some $q_i$ or $r_i$. 

        \begin{figure}
            \centering
            \labellist
                \pinlabel $p$ at 160 200
                \pinlabel $X_1$ at 65 175
                \pinlabel $X_2$ at 270 175
            \endlabellist
            \includegraphics[width=0.3\linewidth]{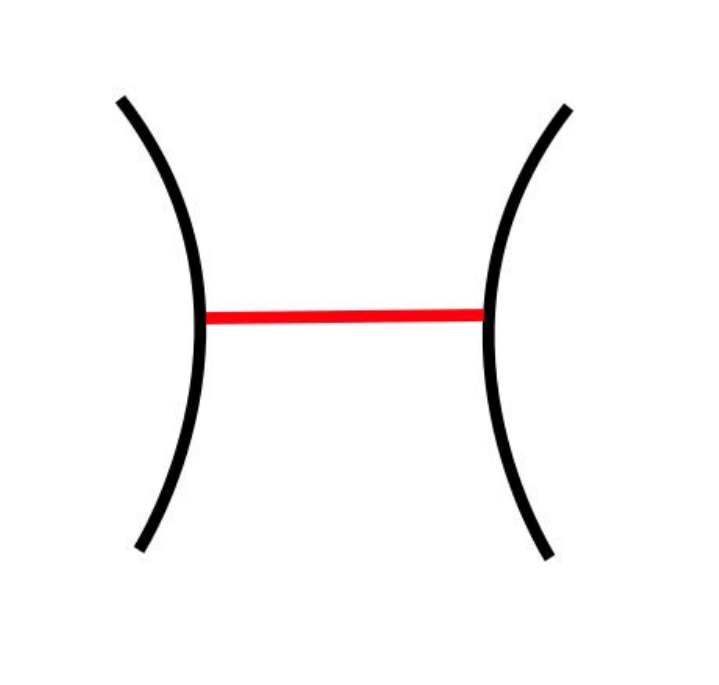}
            \caption{$p$ is red}
            \label{fig:case2pisred}
        \end{figure}
        
        We proceed in an similar fashion.

        \begin{enumerate}
            \item[(R-1)] If $m = 0$ and $n \geq 0$, then $p$ represents the resolution of a removable nugatory crossing.
            \item[(R-2)] If $m > 0$ and $n = 0$, then $q_1$ must be red. 
            We see that this conflicts with the fact that $D_u$ is almost uniform, as $p$ and $q_1$ are on opposite sides of a trivial circle but are both red. 
            \item[(R-3)] If $m > 0$ and $n > 0$, then we proceed as in Case (B-3): we consider crossings in a clockwise fashion.
            If $q_1$ immediately proceeds $p$ in clockwise order, then the argument of Case (R-2) may be used. 
            Suppose otherwise that $r_1, ..., r_j$ succeed $p$ and precede $q_1$.
            If $r_1, ..., r_j$ are all blue arcs, then $q_1$ must be red (otherwise $q_1$ and $p$ are on the same side of $X_1$, but oppositely colored).
            However, this means that $q_1$ and $r_j$ are on the same side of $X_1$, but are oppositely colored.
            
            Hence, some $r_l$ must be red; we may let $r_l$ be the red arc in $D_u$ which is most advanced in clockwise order out of $r_1, ..., r_j$.
            Then $r_l, ..., r_j$ are blue arcs; the contradiction follows as in Case (R-2) once again.
        \end{enumerate}

        We conclude that $D_{u_1}$ is perfectly wrapped. 
        In summary, our process has resulted in a perfectly wrapped resolution, identical to $D_u$ aside from the insertion of $\ww D_{u_1}$ in place of $W_1$ and its interior trivial circles.

        Next, consider $W_2$.
        $W_2$ is a trivial circle in $D_u$ with $\dep(W_2) = 0$.
        The algorithm we performed in $D_u$ has been conducted away from $W_2$, so $W_2$ corresponds to a trivial circle in $D_{u_1}$ (which we will also call $W_2$); furthermore, this trivial circle is not a circle in $\ww D_{u_1}$ by the same reasoning. 
        Therefore, $W_2$ is either type $(0,1)$ or type $(1,0)$. 
        
        If $W_2$ is type $(0,1)$, we repeat our previous process. 
        If $W_2$ is type $(1,0)$, we can use a completely analogous argument (leave $W_2$-exterior arcs the same and switch $W_2$-interior blue arcs to red arcs) to replace $W_2$ and its interior trivial circles by a collection of type $0$ trivial circles.
        We obtain a perfectly wrapped diagram $D_{u_2}$ with a replacement for $W_2$.

        We now repeat this procedure.
        With only finitely many trivial circles to replace, it terminates in a perfectly wrapped resolution $D_v$, where every trivial circle is contained in some collection of type $1$ or type $0$ trivial circles.
        Thus, $D_v$ is perfectly wrapped and uniform, as desired.
    \end{proof}

    The following corollary is implicit in the previous argument, but we state and prove it for emphasis. 

    \begin{corollary}
        \label{cor:slightfurtherextension}
        Suppose that $L$ admits a diagram $D$ without removable nugatory crossings, along with a perfectly wrapped resolution $D_u$. For depth $0$ trivial circles in $W_1, ..., W_n$, suppose that there exist a collection of disks $U_1, ..., U_n$ in the annulus such that
        \begin{enumerate}
            \item $U_i$ contains $W_i$, 
            \item $U_i \cap U_j = \emptyset$ for $i \neq j$, and 
            \item Within each $U_i$, $D_u$ is either almost uniform or uniform.
        \end{enumerate}
        Then there exists a perfectly wrapped uniform resolution of $D$.
    \end{corollary}

    \begin{proof}
        Condition (1) implies that every trivial circle in $D_u$ is contained in some disk. 
        Moreover, if some $W_i$ abuts some $W_j$, they must have the same color of exterior arcs by condition (3). 
        Finally, by condition (2), we can perform the algorithm of Theorem \ref{thm:expandingclass} within every $U_i$ for which $D_u$ is almost uniform. 
        The resulting resolution is perfectly wrapped and uniform. 
    \end{proof}

\section{Classes of links admitting perfectly wrapped uniform resolutions} \label{sec:classes}

    \noindent 
    We may now investigate certain links which adhere to the conditions of Theorem \ref{thm:perfectlywrappeduniform}, and thereby satisfy Conjecture \ref{con:catwrapcon}. 

\subsection{Alternating links} \label{subsec:altlinks}

    As mentioned before, Subsection \ref{subsec:expandingclassofpwur}, and specifically Definition \ref{def:almostuniform}, was motivated by alternating links.
    With Lemma \ref{lma:alternatinglinksalmostuniform}, we see how type $(0,1)$ and $(1,0)$ circles arise naturally as a corollary of alternating phenomena.

    \begin{lemma}
        \label{lma:alternatinglinksalmostuniform}
        Let $D$ be an alternating diagram for an annular link $L$. 
        Let $D_u$ be a resolution of $D$. 
        Then any trivial circle $W$ in $D_u$ is either of type $(0,1)$ or type $(1,0)$.
    \end{lemma}

    \begin{proof}
        $W$ will have a collection $d_1, ..., d_n$ of arcs around it, which we may label in clockwise order (as in $d_{i+1}$ is clockwise from $d_{i}$ around $W$).
        
        We now look at any $d_j$ abutting $W$, and check that $d_{j+1}$ is either of identical color to $d_j$ if it is on the same side of $W$ as $d_j$, or is of opposite color to $d_j$ if it is on the opposite side of $W$ as $d_j$. 
        Suppose that if $d_j$ is exterior, it is red, and if $d_j$ is interior, it is blue; other cases follow completely analogously. 
        The alternating condition then forces almost uniformity, as we see in Figures \ref{fig:djredext} and \ref{fig:djblueint}. This concludes the proof.

        \begin{figure}
            \begin{tikzpicture}
    \begin{scope}[xshift=-3cm, yshift=0cm, scale=.5]
        \draw[thick] (0,0) -- (1,0) -- (2,-1);
        \filldraw[white] (1.5, -.5) circle (.2cm);
        \draw[thick] (1,-1) -- (2,0) -- (3,0) -- (4,1);
        \filldraw[white] (3.5, .5) circle (.2cm);
        \draw[thick] (3,1) -- (4,0) -- (5,0);
    \end{scope}
    \draw[->] (0,0) -- (1,0);
    \def \d{.5}; 
    \begin{scope}[xshift=1.5cm, yshift=0cm, scale=.5]
        \draw[thick] (0,0) -- (5,0);
        \draw[red, ultra thick] (3.5, 0) -- (3.5, \d);
        \draw[blue, ultra thick] (1.5, 0) -- (1.5, -1*\d);
        \node at (3.5, 1) {$d_{j}$};
        \node at (1.5, -1) {$d_{j+1}$};
    \end{scope}
    \begin{scope}[yshift=-2cm]
        \begin{scope}[xshift=-3cm, yshift=0cm, scale=.5]
            \draw[thick] (0,0) -- (1,0) -- (2,1);
            \filldraw[white] (1.5, .5) circle (.2cm);
            \draw[thick] (1,1) -- (2,0) -- (3,0) -- (4,1);
            \filldraw[white] (3.5, .5) circle (.2cm);
            \draw[thick] (3,1) -- (4,0) -- (5,0);
        \end{scope}
        \draw[->] (0,0) -- (1,0);
        \def \d{.5}; 
        \begin{scope}[xshift=1.5cm, yshift=0cm, scale=.5]
            \draw[thick] (0,0) -- (5,0);
            \draw[red, ultra thick] (3.5, 0) -- (3.5, \d);
            \draw[red, ultra thick] (1.5, 0) -- (1.5, 1*\d);
            \node at (3.5, 1) {$d_{j}$};
            \node at (1.5, 1) {$d_{j+1}$};
        \end{scope}
    \end{scope}
\end{tikzpicture}
            \caption{$d_j$ is red and exterior}
            \label{fig:djredext}
        \end{figure}
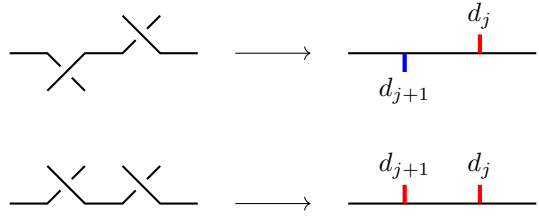
        
        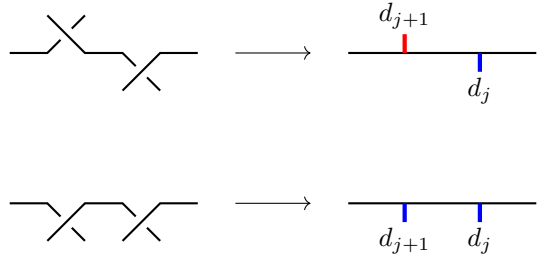
\begin{figure}
            \begin{tikzpicture}
    \begin{scope}[xshift=-.5cm, yshift=0cm, scale=.5, xscale=-1]
        \draw[thick] (3,1) -- (4,0) -- (5,0);
        \filldraw[white] (3.5, .5) circle (.2cm);
        \draw[thick] (1,-1) -- (2,0) -- (3,0) -- (4,1);
        \filldraw[white] (1.5, -.5) circle (.2cm);
        \draw[thick] (0,0) -- (1,0) -- (2,-1);
    \end{scope}
    \draw[->] (0,0) -- (1,0);
    \def \d{.5}; 
    \begin{scope}[xshift=1.5cm, yshift=0cm, scale=.5]
        \draw[thick] (0,0) -- (5,0);
        \draw[blue, ultra thick] (3.5, 0) -- (3.5, -1*\d);
        \draw[red, ultra thick] (1.5, 0) -- (1.5, 1*\d);
        \node at (3.5, -1) {$d_{j}$};
        \node at (1.5, 1) {$d_{j+1}$};
    \end{scope}
    \begin{scope}[yshift=-2cm]
        \begin{scope}[xshift=-.5cm, yshift=0cm, scale=.5, rotate=180]
            \draw[thick] (3,1) -- (4,0) -- (5,0);
            \filldraw[white] (3.5, .5) circle (.2cm);
            \draw[thick] (1,1) -- (2,0) -- (3,0) -- (4,1);
            \filldraw[white] (1.5, .5) circle (.2cm);
            \draw[thick] (0,0) -- (1,0) -- (2,1);
        \end{scope}
        \draw[->] (0,0) -- (1,0);
        \def \d{.5}; 
        \begin{scope}[xshift=1.5cm, yshift=0cm, scale=.5]
            \draw[thick] (0,0) -- (5,0);
            \draw[blue, ultra thick] (3.5, 0) -- (3.5, -1*\d);
            \draw[blue, ultra thick] (1.5, 0) -- (1.5, -1*\d);
            \node at (3.5, -1) {$d_{j}$};
            \node at (1.5, -1) {$d_{j+1}$};
        \end{scope}
    \end{scope}
\end{tikzpicture}
            \caption{$d_j$ is blue and interior}
            \label{fig:djblueint}
        \end{figure}
        
    \end{proof}

    \begin{definition}[Page 10, \cite{HansHomayun-generalizedtait}]
        \label{def:reducedalternatingannular}
        We say a diagram $D$ for an annular link $L$ is a \textit{reduced alternating diagram} if $D$ is alternating and has no removable nugatory crossings.
    \end{definition}

    \begin{corollary}
        \label{cor:alternatinglinkssatisfycwnc}
        Let $L$ be an alternating link. Then $L$ satisfies Conjecture \ref{con:catwrapcon}. 
    \end{corollary}

    \begin{proof}
        In $S^3$, every nugatory crossing is removable, and it is a classical result that every alternating link in $S^3$ admits a reduced alternating diagram.
        From the $S^3$ case, we can conclude that an alternating annular link must admit a reduced alternating annular diagram.
        
        We take $D$ to be such a diagram for $L$.
        By Lemma \ref{lma:existenceofperfectlywrappedres}, there is a perfectly wrapped resolution $D_u$ of $D$, which by Lemma \ref{lma:alternatinglinksalmostuniform} is almost uniform. 
        Then an application of Theorem \ref{thm:expandingclass} to the diagram $D$ gives a perfectly wrapped uniform resolution. 
        Theorem \ref{thm:perfectlywrappeduniform} concludes the proof. 
    \end{proof}

\subsection{Other classes and modifications of links} \label{subsec:otherclasses}

    Beyond alternating links, there are immediate candidates for the existence of perfectly wrapped uniform resolutions which may exhibit arbitrarily non-alternating behavior.
    In particular, we are motivated by plumbed links, arborescent links, and other such classes; these display some of the phenomena we discuss below. 
    
    \begin{corollary}
        \label{cor:braidsandchains}
        Suppose that $T$ is a tangle which is braided outside of disjoint disks $U_1, ..., U_n$, each of which contain a positively or negatively oriented ``chain'' (or $2$-braid) and no other crossings.
        Then the annular closure of $T$, $\hat{T}$, admits a perfectly wrapped uniform resolution. 
    \end{corollary}

    \begin{proof}
        To construct a perfectly wrapped uniform resolution $D_u$ of $\hat{T}$, our first step is to take the braidlike resolution (or, in other words, the oriented resolution of $\hat{T}$ with respect to the braidlike orientation) of $\hat{T}$ outside of $\cup_i U_i$. 
        The corresponding partial resolution is displayed in Figure \ref{fig:partialchainresolution}.

        \begin{figure}
            \centering
            \begin{tikzpicture}[scale=.5]
\foreach \x in {1,2,3,6,7,8}{
    \draw[thick] (\x,0) -- (\x,6);
}
\filldraw[white] (1.5,4) circle (1cm);
\draw[dashed] (1.5,4) circle (1cm);
\node at (1.5,4) {$U_i$};
\filldraw[white] (7.5,2) circle (1cm);
\draw[dashed] (7.5,2)  circle (1cm);
\node at (7.5,2)  {$U_j$};
\node at (2,7) {$\vdots$};
\node at (2,-.5) {$\vdots$};
\node at (7,7) {$\vdots$};
\node at (7,-.5) {$\vdots$};
\node at (4.5, 4) {$\cdots$};
\node at (4.5, 2) {$\cdots$};
\end{tikzpicture}
            \caption{A partial resolution of $\hat{T}$ outside of $\cup_i U_i$}
            \label{fig:partialchainresolution}
        \end{figure}
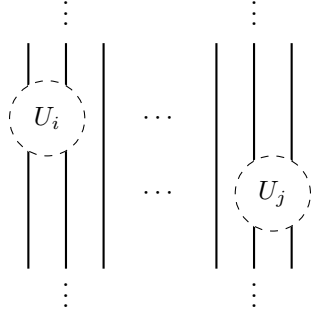

        For our complete resolution of $\hat{T}$, we work by cases on the orientation of the chain $C_i$ in $U_i$. 
        If $C_i$ is in line with the braid, then we can consider it as part of the braid, and we take the braidlike resolution of $C_i$ as shown in Figure \ref{fig:braidlikeresolutionofchain}. 
        If $C_i$ is not in line with the braid, then we resolve chainlike, as in Figure \ref{fig:chainlikeres}. 

        \begin{figure}
            \centering
            \begin{tikzpicture}[scale=.75]
\begin{scope}
    \draw (0,0) -- (1,1);
    \filldraw[white] (.5,.5) circle (.2cm);
    \draw (0,1) -- (1,0);
\end{scope}
\begin{scope}[yshift=2cm]
    \draw (0,0) -- (1,1);
    \filldraw[white] (.5,.5) circle (.2cm);
    \draw (0,1) -- (1,0);
\end{scope}
\node at (.5,1.5) {$\vdots$};
\foreach \y in {0,4}{
    \draw (0,\y) -- (0,\y-1);
    \draw (1,\y) -- (1,\y-1);
}
\draw[->] (2,1.5) --  (3,1.5);
\begin{scope}[xshift=4cm]
    \foreach \y in {0,3}{
        \draw (0,\y+1) -- (0,\y-1);
        \draw (1,\y+1) -- (1,\y-1);
    }
    \draw[blue, thick] (0,.5) -- (1,.5);
    \draw[blue, thick] (0,2.5) -- (1,2.5);
    \node at (.5,1.5) {$\vdots$};
\end{scope}

\end{tikzpicture}
            \caption{The braidlike resolution of a negative chain}
            \label{fig:braidlikeresolutionofchain}
        \end{figure}
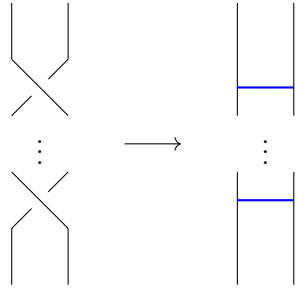
        
        \begin{figure}
            \centering
            \begin{tikzpicture}[scale=.8]
\foreach \x in {0,1,3,4}{
    \begin{scope}[xshift=\x cm]
        \draw[thick] (0,0) -- (1,1);
        \filldraw[white] (.5,.5) circle (.2cm);
        \draw[thick] (0,1) -- (1,0);
    \end{scope}
}
\node at (2.5,.5) {$\cdots$};
\foreach \y in {0,2}{
    \draw[thick] (0,\y) -- (0,\y-1);
    \draw[thick] (5,\y) -- (5,\y-1);
}
\draw[->] (6,.5) --  (7,.5);
\begin{scope}[xshift=8cm]
    \foreach \x in {1,4}{
        \draw[thick] (\x,.5) circle (.4cm);
        \draw[blue, very thick] (\x-.4, .5) -- (\x-1,.5);
        \draw[blue, very thick] (\x+.4, .5) -- (\x+1,.5);
    }
    \node at (2.5,.5) {$\cdots$};
    \foreach \x in {0,5}{
        \draw[thick] (\x,-1) -- (\x,2);
    }
\end{scope}
\end{tikzpicture}
            \caption{The chainlike resolution of a negative chain}
            \label{fig:chainlikeres}
        \end{figure}
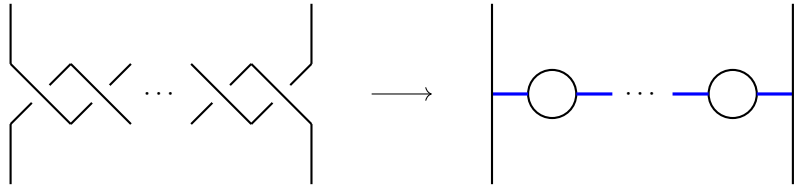

        The resulting complete resolution $D_u$ is exactly wrapped by construction. 
        Moreover, every trivial circle in $D_u$ arises as a result of some chainlike resolution since $\hat{T}$ is resolved braidlike outside of the $U_i$ disks.
        It follows by inspection that $D_u$ is perfectly wrapped and uniform.
    \end{proof}
    
    \begin{corollary}
        \label{cor:circlecable}
        Consider an $(n,n)$-tangle $T_{k}$ of the type displayed in Figure \ref{fig:tktangle}, where the unlink component encircles any $k$ strands on the trivial $n$-braid in the pictured manner. 
        
        \begin{figure}
            \centering
            \includegraphics[width=0.4\linewidth]{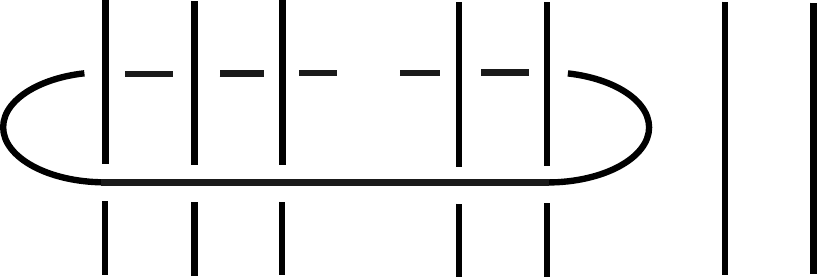}
            \caption{The $(n,n)$-tangle $T_k$}
            \label{fig:tktangle}
        \end{figure}
        
        Let $L$ be the link obtained by vertically stacking $T_{k_1}, ..., T_{k_m}$ along with $n$-braids $B_1, ..., B_l$ in any order, and taking the annular closure. 
        Then $L$ admits at least $2^{m}$ perfectly wrapped uniform resolutions. 
    \end{corollary}

    \begin{proof}
        Resolve each $B_i$ braidlike, and take any one of the  resolutions of each $T_{k_j}$ displayed in Figures \ref{fig:circlecable0} and \ref{fig:circlecable1}. 
        
        \begin{figure}
            \centering
            \includegraphics[width=0.4\linewidth]{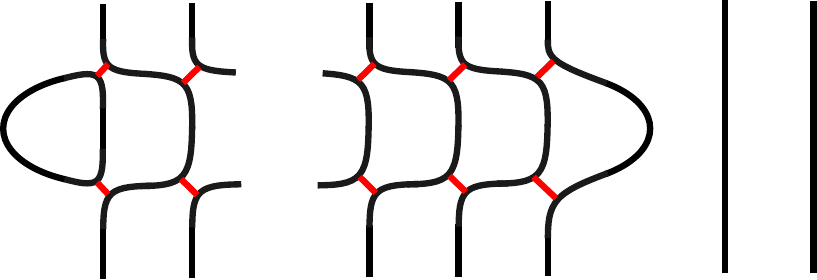}
            \caption{The $0$-type perfectly wrapped uniform resolution}
            \label{fig:circlecable0}
        \end{figure}

        \begin{figure}
            \centering
            \includegraphics[width=0.4\linewidth]{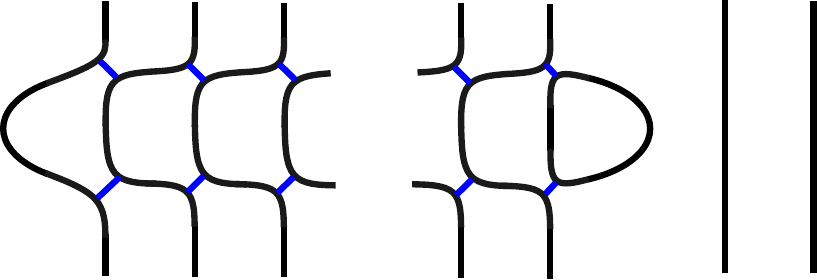}
            \caption{The $1$-type perfectly wrapped uniform resolution}
            \label{fig:circlecable1}
        \end{figure}

        We see that the resulting exactly wrapped resolution $D_u$ is perfectly wrapped and uniform, as every trivial circle in $D_u$ is of the form displayed in either Figure \ref{fig:circlecable0} or Figure \ref{fig:circlecable1}.
    \end{proof} 

    In the spirit of the previous class, we can examine the operation of cabling in the setting of perfectly wrapped uniform resolutions.

    \begin{corollary}
        \label{cor:cablingop}
        Let $D$ be a diagram for an annular link $L$ which admits a perfectly wrapped uniform resolution $D_u$.
        Let $D^n$ be the blackboard-framed $n$-cabling of $D$, representing an annular link $L^n$.
        Then $D^n$ admits a perfectly wrapped uniform resolution $D_v^n$. 
    \end{corollary}

    \begin{proof}
        At any given crossing $c$ in $D$, we see a grid pattern of crossings in $D^n$. 
        We then resolve each crossing in this grid exactly as we did for $c$.
        Figure \ref{fig:3cable0res} demonstrates this algorithm with a $0$-resolution of $c$ in the case $n = 3$.
        
        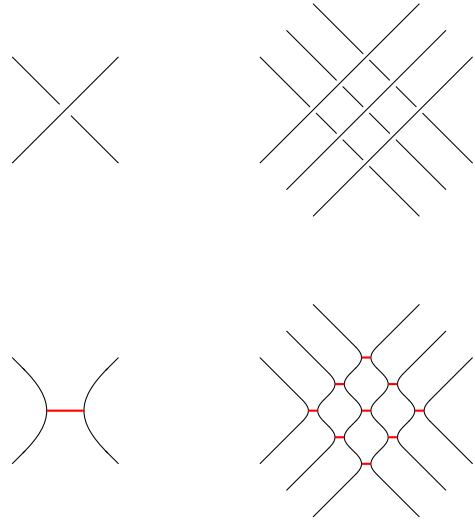
\begin{figure}
            \centering
            \begin{tikzpicture}
\def \d{3};
\def \r{.1};
\def \R{.4};
\def \RR{.8};
\def \m{.22}; 
\def \mR{\m*\R}; 
\def \mRR{\m*\RR};
\def \s{.25}; 
\begin{scope}[rotate=45]
    \draw (0,-1) -- (0,1);
    \filldraw[white] (0,0) circle (\r cm);
    \draw (-1,0) -- (1,0);
\end{scope}
\begin{scope}[xshift=4cm, scale=.5, rotate=45]
    \foreach \x in {-1,0,1}{
        \draw (\x,-\d) -- (\x, \d);
        \foreach \y in {-1,0,1}{
            \filldraw[white] (\x,\y) circle (\r cm);
        }
    }
    \foreach \y in {-1,0,1}{
        \draw (-\d, \y) -- (\d, \y);
    }
\end{scope}
\begin{scope}[yshift=-4cm]
    \begin{scope}[rotate=45]
        \begin{scope}
            \draw (0,-1) -- (0,1);
            \draw (-1,0) -- (1,0);
            \filldraw[white] (0,0) circle (3*\s cm);
        \end{scope}
        \draw (-\RR,0) .. controls (-.25*\RR, 0) and (0, .25*\RR) ..  (0,\RR);
        \draw (0,-\RR) .. controls (0, -.25*\RR) and (.25*\RR,0) .. (\RR,0);
        \draw[thick, red] (-\mRR, \mRR) -- (\mRR, -\mRR);
    \end{scope}
    \begin{scope}[xshift=4cm, scale=.5, rotate=45]
        \foreach \x in {-1,0,1}{
            \draw (\x,-\d) -- (\x, \d);
        }
        \foreach \y in {-1,0,1}{
            \draw (-\d, \y) -- (\d, \y);
        }
        \foreach \x in {-1,0,1}{
            \foreach \y in {-1,0,1}{
                \filldraw[white] (\x,\y) circle (1.5*\s cm);
                \begin{scope}[xshift=\x cm, yshift=\y cm]
                    \draw (-\R,0) .. controls (-.25*\R, 0) and (0, .25*\R) ..  (0,\R);
                    \draw (0,-\R) .. controls (0, -.25*\R) and (.25*\R,0) .. (\R,0);
                    \draw[thick, red] (-\mR, \mR) -- (\mR, -\mR);
                \end{scope}
            }
        }
    \end{scope}
\end{scope}
\end{tikzpicture}
            \caption{Cabling resolution applied to a $3$-cabling}
            \label{fig:3cable0res}
        \end{figure}

        $D_v^n$ is the blackboard-framed $n$-cabling of $D_u$.
        If a trivial circle in $D_u$ is abutted by all $0$-resolutions, then each of the $n$ corresponding trivial circles in $D_v^n$ are abutted by all $0$-resolutions. 
        The same holds for $1$-resolutions. 
        Hence, $D_u$ is uniform.

        Moreover, 
            \[\wrap(D_v^n) = n \cdot \wrap(D_u) = n \cdot \wrap(L) = \wrap(L^n).\]
        Hence, $D_v^n$ is exactly wrapped.

        Finally, if switching a crossing in $D_u$ enacts some type of cobordism, then switching any crossing in the corresponding grid in $D_v^n$ either enacts that same type of cobordism or results in a merge of two adjacent circles. 
        Therefore, as $D_u$ is insulated, $D_v^n$ is insulated as well.         
    \end{proof}

    \begin{corollary}
        \label{cor:loopinsertion}
        Let $L$ be a link with a diagram $D$ that admits a perfectly wrapped uniform resolution $D_u$. 
        Suppose that $D'$ is a diagram (for a link $L'$), which is identical to $D$ outside of disjoint disks $U_1, ..., U_n$. 
        Each $U_i$ contains a single linked loop, or ``earring,'' as shown in Figure \ref{fig:oncewrappedunlinkcomp}, and no other crossings.
        Then $D'$ admits a perfectly wrapped uniform resolution.
        
        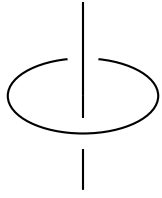
\begin{figure}
            \centering
            \begin{tikzpicture}[yscale=.5]
\draw[thick] (0,-1) -- (0,1.5);
\filldraw[white] (0,.5) circle (.4cm);
\draw[thick] (0,1.5) circle (1cm);
\filldraw[white] (0,2.5) circle (.2cm);
\draw[thick] (0,1.5) -- (0,4);
\end{tikzpicture}
            \caption{A linked loop, or earring}
            \label{fig:oncewrappedunlinkcomp}
        \end{figure}
    
    \end{corollary}

    \begin{proof}
        We construct a perfectly wrapped uniform resolution $D_v'$ by first considering a partial resolution of $D'$. 
        Outside of $\bigcup_i U_i$, we resolve each crossing in $D'$ as in $D_u$. 
        Now within a given $U_i$, consider two resolutions of the linked loop, displayed in Figure \ref{fig:tworesolutionsoflinkedloops}. 
        We note that for each linked loop $O_i$ in $U_i$, there is a choice of whether to produce a type $0$ or type $1$ trivial circle.

        \begin{figure}
            \centering
            \includegraphics[width=0.5\linewidth]{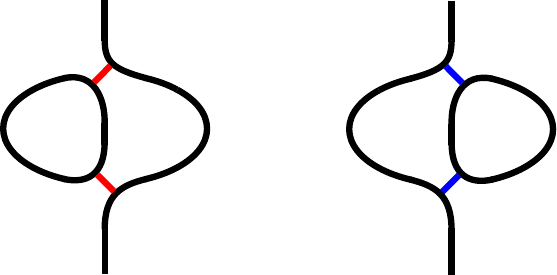}
            \caption{Two resolutions of a linked loop}
            \label{fig:tworesolutionsoflinkedloops}
        \end{figure}

        In our partial resolution of $D'$, the arc passing through a loop $O_i$ may be, thus far, part of some nontrivial circle.
        In this case, we can use either choice of pictured resolution for $O_i$. 
        Alternatively, if the arc passing through $O_i$ is thus far part of some trivial circle, this circle will abut only $0$-resolutions or only $1$-resolutions in the partial resolution. 
        Hence, we resolve $O_i$ in the manner that allows for compatible type. 
        
        Our resulting complete resolution $D_v'$ is functionally identical to $D_u$, but within each $U_i$ we insert a trivial circle of the kind in Figure \ref{fig:tworesolutionsoflinkedloops}. 
        These trivial circles will be of compatible type with the circles they abut. 
        Hence, $D_v'$ is uniform. 
        Moreover, we see immediately that $D_v'$ is perfectly wrapped within each $U_i$. 
        Given that $D_v'$ is identical to $D_u$ outside of $\bigcup_i U_i$, it follows that $D_v'$ is perfectly wrapped. 
    \end{proof}    
    
\bibliographystyle{alpha}
\bibliography{main}

\begin{thebibliography}{GLW18}

\bibitem[Akh23]{Akhmechet-equivariant}
Rostislav Akhmechet.
\newblock Equivariant annular {K}hovanov homology.
\newblock {\em J. Knot Theory Ramifications}, 32(2):Paper No. 2350002, 31, 2023.

\bibitem[APS04]{APS-khovanovhomologyoversurfaces}
Marta~M. Asaeda, J\'ozef~H. Przytycki, and Adam~S. Sikora.
\newblock Categorification of the {K}auffman bracket skein module of {$I$}-bundles over surfaces.
\newblock {\em Algebr. Geom. Topol.}, 4:1177--1210, 2004.

\bibitem[BKS23]{HansHomayun-generalizedtait}
Hans~U. Boden, Homayun Karimi, and Adam~S. Sikora.
\newblock Adequate links in thickened surfaces and the generalized {T}ait conjectures.
\newblock {\em Algebr. Geom. Topol.}, 23(5):2271--2308, 2023.

\bibitem[BN05]{BN-tanglesandcobs}
Dror Bar-Natan.
\newblock Khovanov's homology for tangles and cobordisms.
\newblock {\em Geom. Topol.}, 9:1443--1499, 2005.

\bibitem[BPW19]{BPW-AKhq}
Anna Beliakova, Krzysztof~K. Putyra, and Stephan~M. Wehrli.
\newblock Quantum link homology via trace functor {I}.
\newblock {\em Invent. Math.}, 215(2):383--492, 2019.

\bibitem[GLW18]{GLW-sl2action}
J.~Elisenda Grigsby, Anthony~M. Licata, and Stephan~M. Wehrli.
\newblock Annular {K}hovanov homology and knotted {S}chur-{W}eyl representations.
\newblock {\em Compos. Math.}, 154(3):459--502, 2018.

\bibitem[GN14]{GrElYi-braiddetection}
J.~Elisenda Grigsby and Yi~Ni.
\newblock Sutured {K}hovanov homology distinguishes braids from other tangles.
\newblock {\em Math. Res. Lett.}, 21(6):1263--1275, 2014.

\bibitem[HP95]{HostePrz-skeinmoduleofwhiteheadmanifolds}
Jim Hoste and J\'ozef~H. Przytycki.
\newblock The {$(2,\infty)$}-skein module of {W}hitehead manifolds.
\newblock {\em J. Knot Theory Ramifications}, 4(3):411--427, 1995.

\bibitem[Kho00]{Khovanov-original}
Mikhail Khovanov.
\newblock A categorification of the {J}ones polynomial.
\newblock {\em Duke Math. J.}, 101(3):359--426, 2000.

\bibitem[Kim21]{Kim-caltechsl2thesis}
Juhyun Kim.
\newblock {\em Annular {L}inks with {$\frak{Sl}_2$}-{I}rreducible {A}nnular {K}hovanov {H}omology}.
\newblock ProQuest LLC, Ann Arbor, MI, 2021.
\newblock Thesis (Ph.D.)--California Institute of Technology.

\bibitem[Lee05]{Lee-leehom}
Eun~Soo Lee.
\newblock An endomorphism of the {K}hovanov invariant.
\newblock {\em Adv. Math.}, 197(2):554--586, 2005.

\bibitem[Mar23]{Martin-wrapping}
Gage Martin.
\newblock Annular {K}hovanov homology and meridional disks.
\newblock {\em J. Knot Theory Ramifications}, 32(2):Paper No. 2250088, 14, 2023.

\bibitem[Prz91]{Prz-skein}
J\'ozef~H. Przytycki.
\newblock Skein modules of {$3$}-manifolds.
\newblock {\em Bull. Polish Acad. Sci. Math.}, 39(1-2):91--100, 1991.

\bibitem[Ras10]{ras-s}
Jacob Rasmussen.
\newblock Khovanov homology and the slice genus.
\newblock {\em Invent. Math.}, 182(2):419--447, 2010.

\bibitem[Rob13]{Roberts-akhfloer}
Lawrence~P. Roberts.
\newblock On knot {F}loer homology in double branched covers.
\newblock {\em Geom. Topol.}, 17(1):413--467, 2013.

\bibitem[Xie21]{Yi-instantakh}
Yi~Xie.
\newblock Instantons and annular {K}hovanov homology.
\newblock {\em Adv. Math.}, 388:Paper No. 107864, 51, 2021.

\end{thebibliography}

\end{document}